\documentclass[11pt]{amsart}
\usepackage{graphicx,subfigure,epsfig}
\usepackage{mathtools}

\usepackage{amssymb,amsmath}
\usepackage{caption}
\usepackage{float}
\usepackage[usenames, dvipsnames]{color}
\usepackage{bm}
\usepackage{enumerate}
\usepackage{float}
\newtheorem{theorem}{Theorem}[section]
\newtheorem{definition}{Definition}[section]
\newtheorem{lemma}{Lemma}
\newtheorem{remark}{Remark}

\newtheorem{corollary}{Corollary}
\newtheorem{example}{Example}
\newtheorem{proposition}{Proposition}

\begin{document}
\title{Arc shift number and region arc shift number for virtual knots}

\author{ K. Kaur} 
\address{Kirandeep Kaur\\ Department of Mathematics\\ Indian Institute of Technology Ropar\\
 Nangal Road, Rupnagar, Punjab 140001, INDIA } 
\email{kirandeep.kaur@iitrpr.ac.in}

\author{ A. Gill} 
\address{Amrendra Gill\\ Department of Mathematics\\ Indian Institute of Technology Ropar\\
 Nangal Road, Rupnagar, Punjab 140001, INDIA } 
\email{amrendra.gill@iitrpr.ac.in}

\author{ M. Prabhakar } 
\address{Madeti Prabhakar \\Department of Mathematics\\ Indian Institute of Technology Ropar\\
 Nangal Road, Rupnagar, Punjab 140001, INDIA}  
\email{prabhakar@iitrpr.ac.in}

\begin{abstract}
In this paper, we formulate a new local move on virtual knot diagram, called arc shift move. Further, we extend it to another local move called region arc shift defined on a region of a virtual knot diagram. We establish that these arc shift and region arc shift moves are unknotting operations by showing that any virtual knot diagram can be turned into trivial knot using arc shift (region arc shift) moves. Based upon the arc shift move and region arc shift move, we define two virtual knot invariants, arc shift number and region arc shift number respectively.
\end{abstract}

\makeatletter{\renewcommand*{\@makefnmark}{}
\makeatother}
\maketitle
\section{Introduction} Virtual knot theory introduced by L.H. Kauffman \cite{kauffman1999virtual} extends classical knot theory to more general study of knots in the thickened surfaces of higher genus. Classical knot theory as a subclass deals with knots in thickened sphere. Unlike in classical knots, crossing change fails to be an unknotting operation for virtual knots. In this paper, we propose a new local move on a virtual knot diagram, called \emph{arc shift}, which happens to be an unknotting operation for virtual knots assisted by generalized Reidemeister moves. Gaining motivation from the region crossing change defined in \cite{shimizu2014region}, we extend the notion of \emph{arc shift} to \emph{region arc shift} defined on a region in virtual knot diagram which turns out to be an unknotting operation too. Minimum number of \emph{arc shift}(respectively, \emph{region arc shift}) moves needed to unknot a virtual knot $K$ is defined as \emph{arc shift number}(respectively, \emph{region arc shift number}) of $K$ denoted by $A(K)$(respectively, $R(K)$). It is well known that every virtual knot diagram can be made trivial via a sequence of forbidden moves and generalized Reidemeister moves as shown in \cite{kanenobu2001forbidden,nelson2001unknotting}. The forbidden number $F(K)$ in \cite{crans2015forbidden} is defined as the minimum number of forbidden moves needed to deform $K$ into a trivial knot. We prove that for a virtual knot $K$, $R(K)\leq F(K)$ and provide an explicit example where the inequality $R(K)\leq F(K)$ is indeed a strict inequality.\\
This paper is organized as follows. In Section 2, we briefly recall the definition of virtual knots including Gauss diagrams and forbidden moves. Section 3 defines the arc shift move, discussing its Gauss diagram version and properties of the arc shift move. In Section 4, we prove arc shift as an unknotting operation and discuss bound on the arc shift number. Lastly in Section 5, we extend arc shift to region arc shift and compare it with forbidden moves.

\section{Prelimenaries} L.H. Kauffman \cite{kauffman1999virtual} introduced virtual knot theory with a motivation to study knots embedded in thickened surfaces of arbitrary genus and also to provide knot theory a completely combinatorial territory dealing with Gauss diagrams and Gauss codes. A virtual knot diagram is a $4$-regular (each node having degree four) planar graph with extra structure on its nodes. The extra structure includes two types of crossings at nodes with one being the classical crossing adhering over (under) information and other one called the virtual crossing. A virtual crossing is indicated by a small circle around the node with no sense of over and under information (See Fig.~\ref{CROSSING}).
\begin{figure}[!ht]
\centering
\includegraphics[scale=.4]{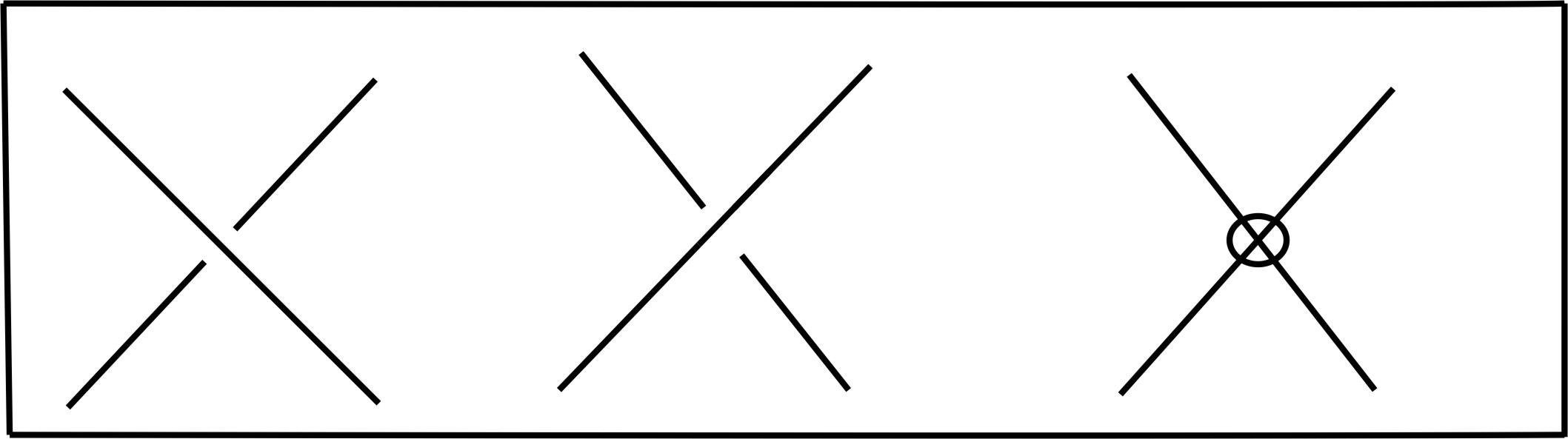}
\caption{Classical and Virtual crossings }
  \label{CROSSING}
\end{figure}

\noindent Two virtual knot diagrams are declared equivalent if they are related by a sequence of generalized Reidemeister moves depicted in Fig.~\ref{GRDMOVE}. Virtual knots are defined as equivalence classes of virtual knot diagrams modulo generalized Reidemeister moves.

\begin{figure}[H]
\centering
\subfigure
{
\includegraphics[width=6cm,height=3cm]{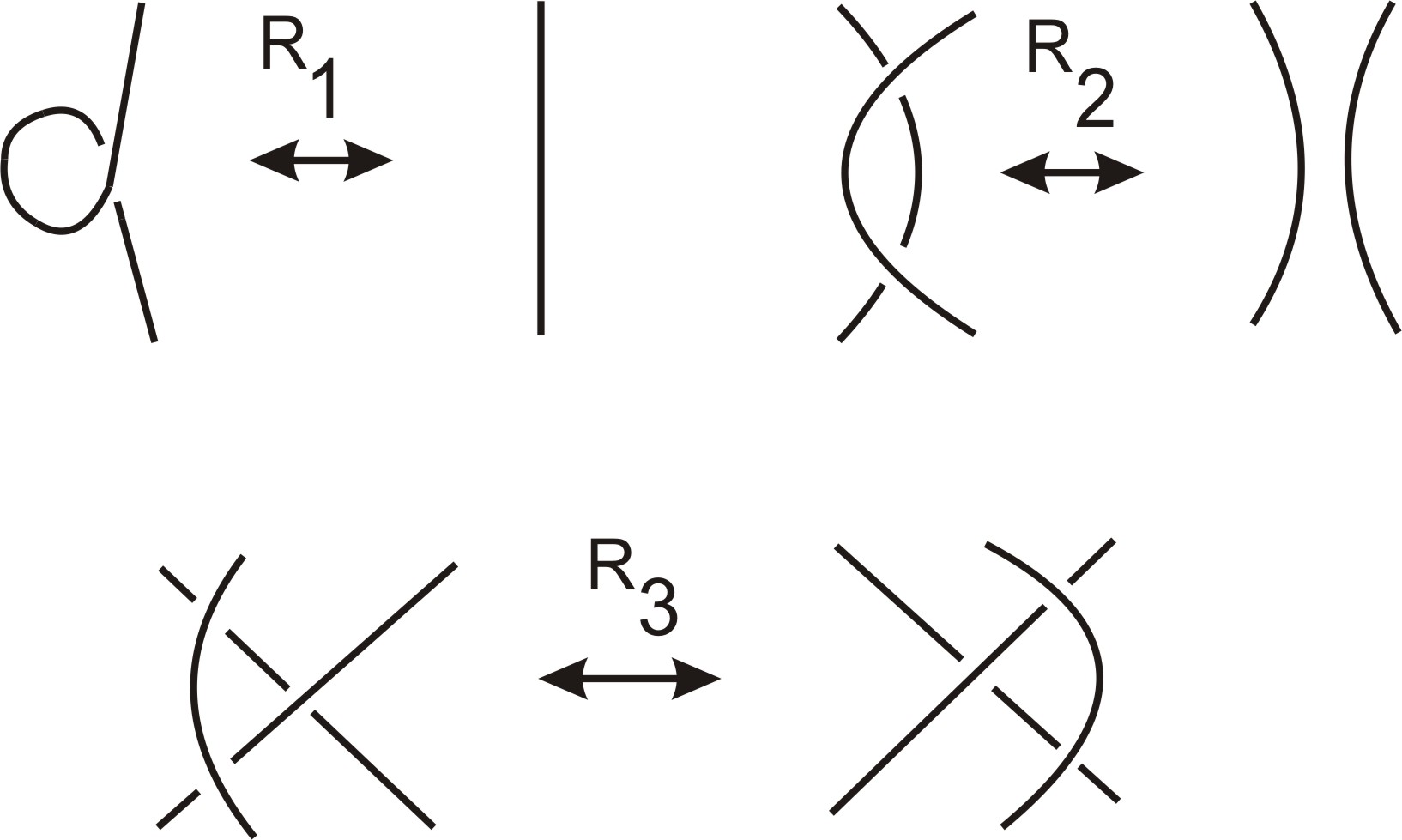}
}
\subfigure
{
\includegraphics[width=6cm,height=3cm]{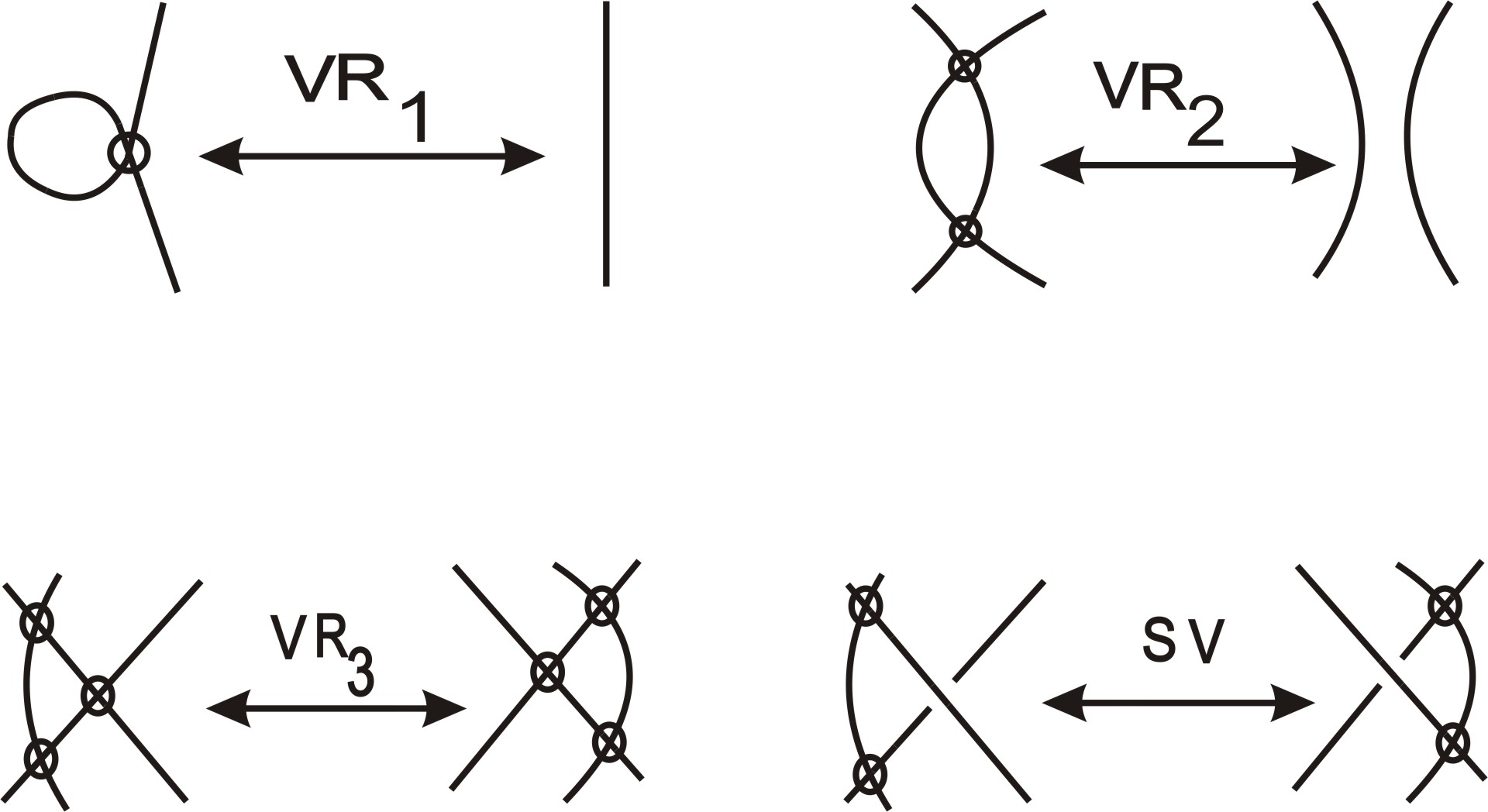}
}
\caption{Generalized Reidemeister moves}
  \label{GRDMOVE}
\end{figure}

\noindent As a consequence of virtual Reidemeister moves a segment in the virtual knot diagram consisting of virtual crossings only can be freely moved in the plane. While moving such segment of the knot, we keep the end points fixed such that all the new places it crosses the diagram transversally are marked as virtual crossings. Moving such a segment is termed as \textbf{Detour move} (Fig.~\ref{DETOUR}).
\begin{figure}[H]
\centering
  \includegraphics[width=8cm,height=2cm]{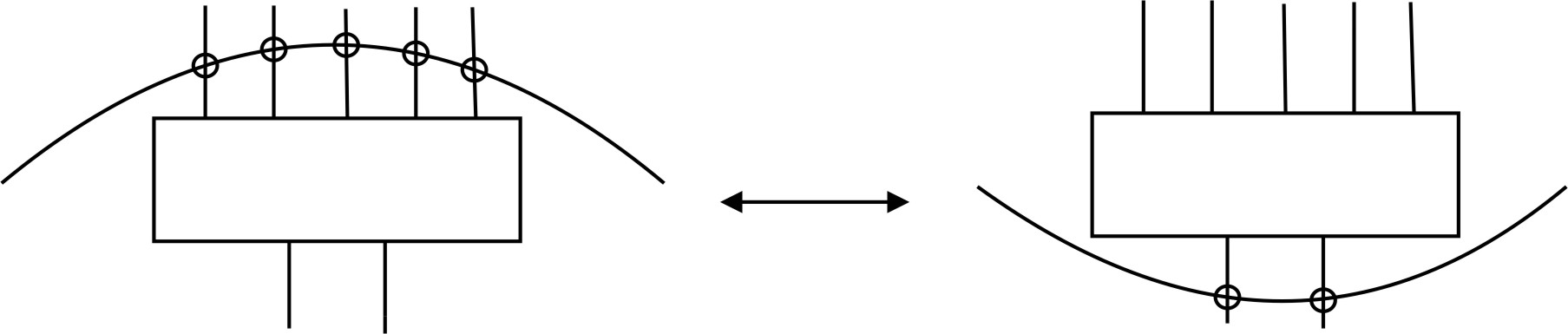}

  \caption{Detour Move }
  \label{DETOUR}
\end{figure}

\noindent Sign of a classical crossing $c$ also known as local writhe of $c$ is defined as shown in the Fig.~\ref{SIGN1}.
\begin{figure}[H]
\centering
  \includegraphics[width=6cm,height=2cm]{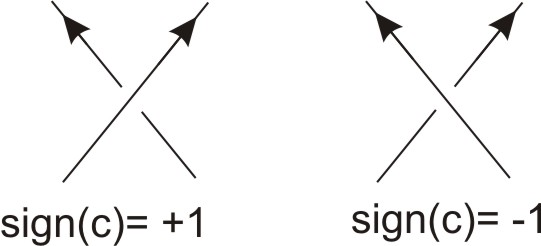}

  \caption{Local writhe or sign of a crossing }
  \label{SIGN1}
\end{figure}

\noindent In \cite{goussarov2000finite} an approach to virtual knot theory goes via Gauss diagrams.
\begin{definition}
Gauss diagram $G(D)$ corresponding to a classical(virtual) knot diagram $D$ is an oriented circle with a base point where each classical crossing is marked two times with respect to overpass and underpass. Two markings are then joined by an arrow (chord) oriented from overpass to underpass with a sign attached to each arrow equal to the local writhe of the corresponding crossing (Fig.~\ref{GDFIG82}).
\end{definition}
\begin{figure}[H]
\centering
  \includegraphics[width=7cm,height=3cm]{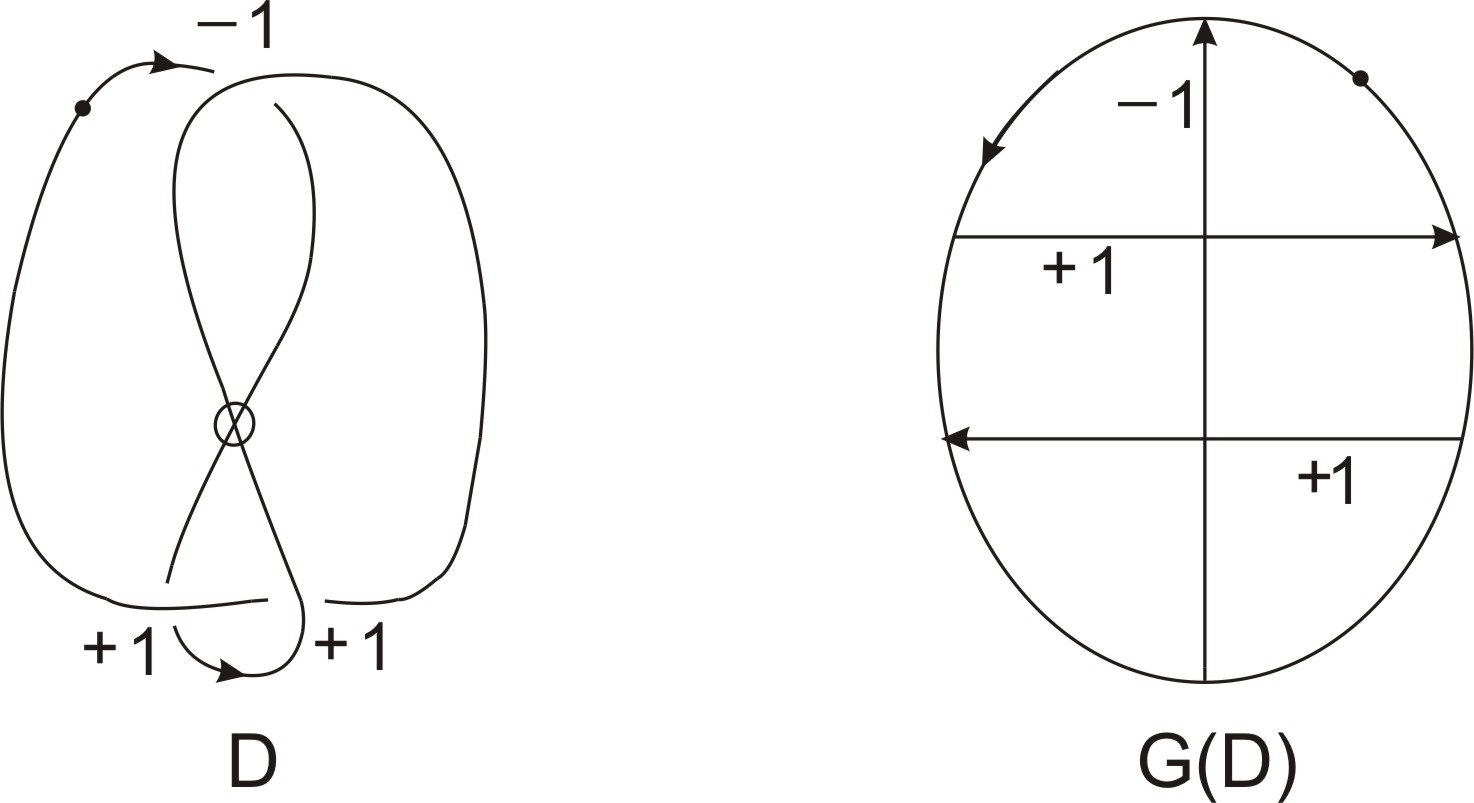}
 \caption{Gauss diagram corresponding to virtual figure eight knot }
  \label{GDFIG82}
\end{figure}
Analogous version of Reidemeister moves on Gauss diagrams are shown in Fig.~\ref{GDRMOVES1}. Virtual moves do not affect Gauss diagrams since virtual crossings are not accounted in Gauss diagrams. An alternative way to define virtual knots is by considering equivalence classes of Gauss diagrams modulo the moves shown in Fig.~\ref{GDRMOVES1}. 
\begin{figure}[H]
\centering
  \includegraphics[width=11cm,height=3.5cm]{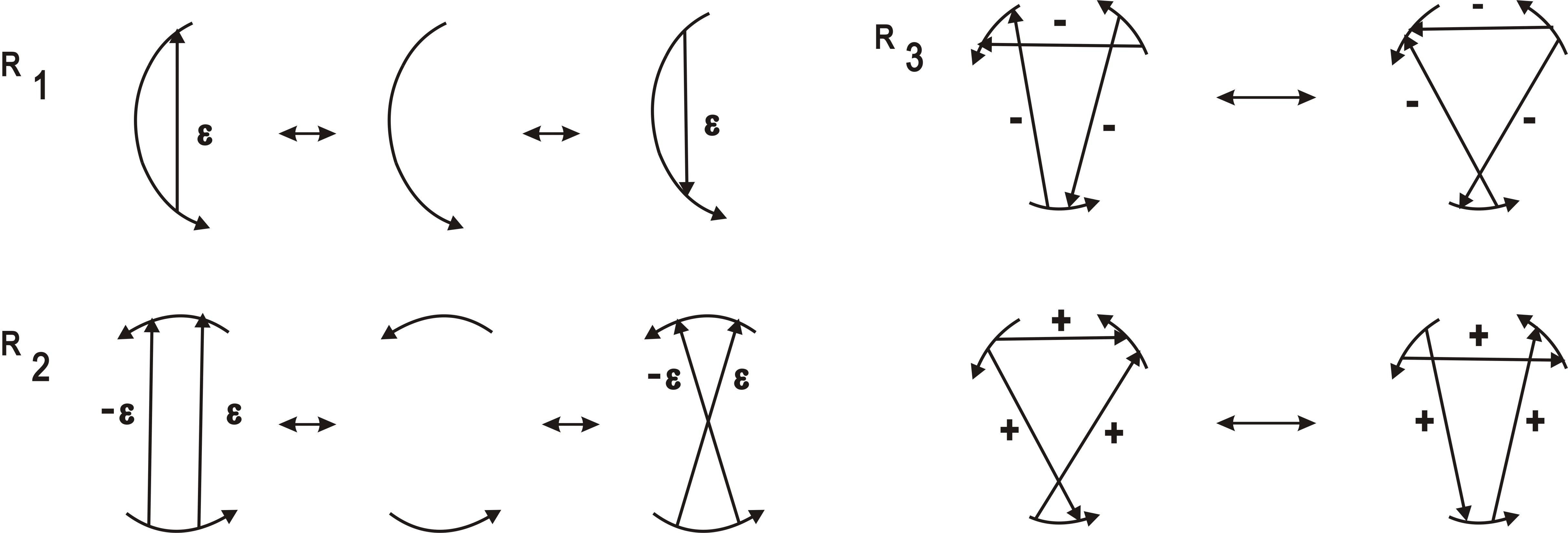}
 \caption{Reidemeister moves on Gauss diagram  }
  \label{GDRMOVES1}
\end{figure}
\noindent Two moves listed in Fig.~\ref{FORBIDDEN}($a$) are known as \textbf{forbidden moves}. Gauss diagram for forbidden moves are shown in Fig.~\ref{FORBIDDEN}($b$). Terminology inspires from the fact that if allowed, these moves will leave whole of virtual knot theory trivial, i.e., any virtual knot can be turned trivial using forbidden moves. This fact is proved in \cite{kanenobu2001forbidden,nelson2001unknotting,goussarov2000finite}.
\begin{figure}[H]
\centering
  \includegraphics[width=9cm,height=3cm]{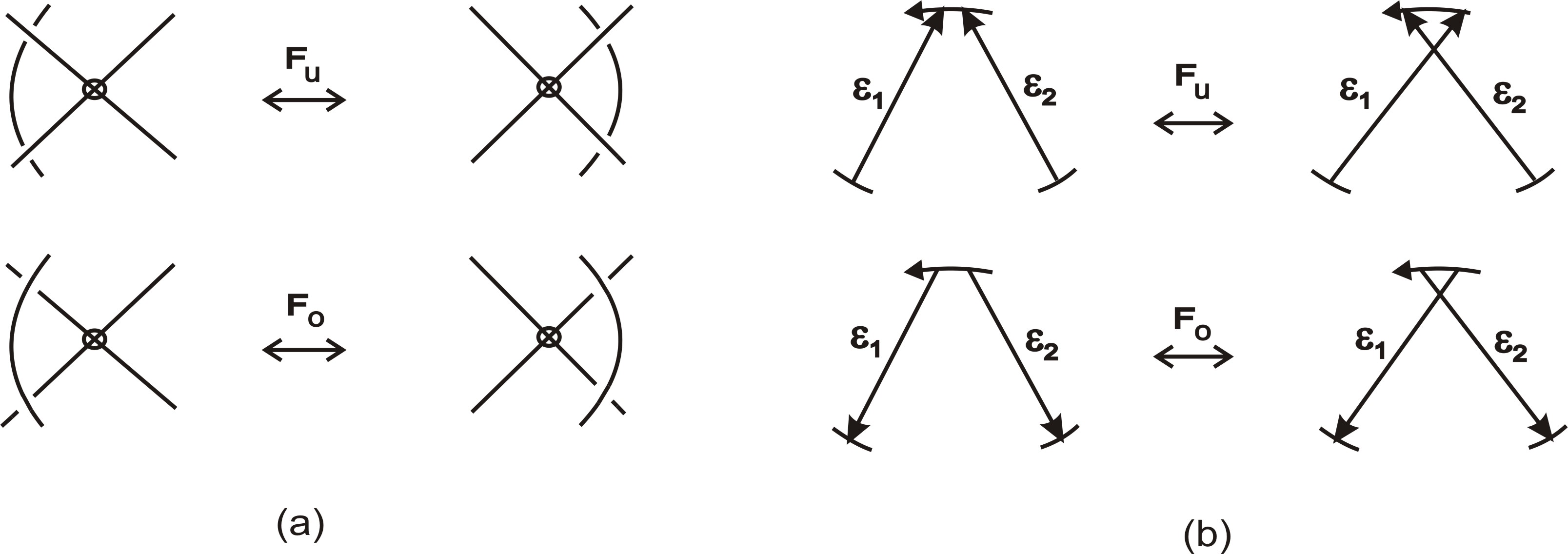}
 \caption{Forbidden moves $F_h$ and $F_t$}
  \label{FORBIDDEN}
\end{figure}
\section{Arc shift move}
\begin{definition}{\rm
In a virtual knot diagram $D$, we define an \emph{arc}, say $(a,b)$ as the segment passing through exactly one pair of crossings (classical/virtual)
($c_{1}$,$c_{2}$) with $a$ incident to $c_1$ and $b$ incident to $c_2$. In Fig.~\ref{VT}, arc $(a,b)$ passes through crossings ($c_{1}$,$c_{2}$) and arc $(e,f)$ passes through crossings ($c_{2}$,$c_{3}$). \\
\begin{figure}[H]
\centering
  \includegraphics[width=3cm,height=2.7cm]{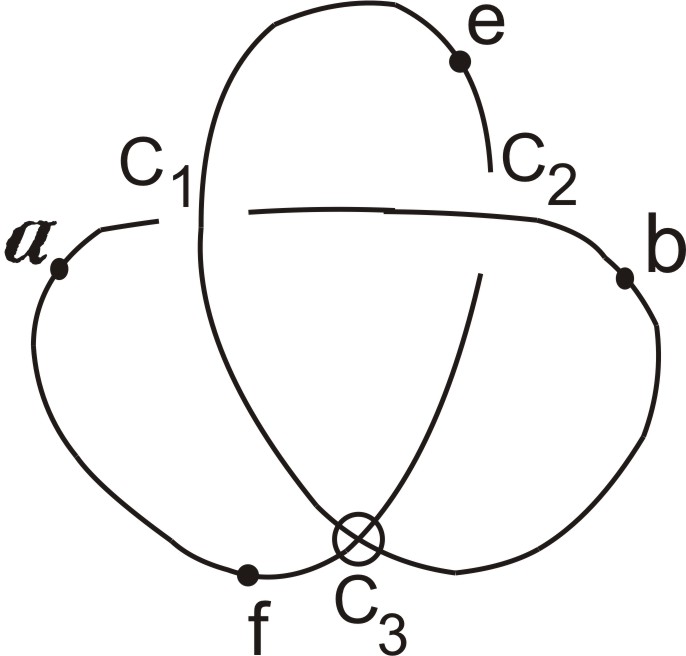}
 \caption{Arc $(a,b)$ and $(e,f)$ }
  \label{VT}
\end{figure}

}\end{definition}
\begin{definition}{\rm
Two \emph{arcs} $(a,b)$ and $(c,d)$ in a virtual knot diagram $D$ are said to be equivalent if they pass through same pair of crossings ($c_{1}$,$c_{2}$) and the segment common in both $(a,b)$ and $(c,d)$ is again an arc passing through the same pair of crossings ($c_{1}$,$c_{2}$). For example arcs $(a,b)$ and $(c,d)$ depicted in Fig.~\ref{EQARCS2} are equivalent arcs.}
\end{definition}
\begin{figure}[H]
\centering
  \includegraphics[width=3.5cm,height=3cm]{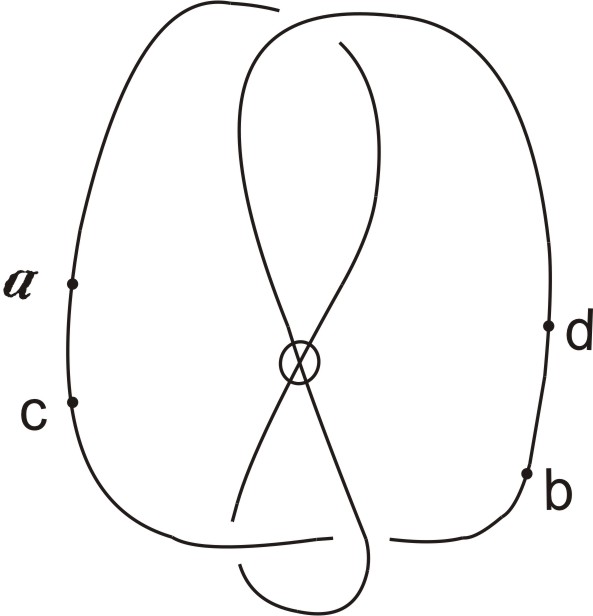}
 \caption{ Equivalent Arcs $(a,b)$ and $(c,d)$  }
  \label{EQARCS2}
\end{figure}

\noindent \textit{Arc shift move:} In a virtual knot diagram $D$, let $(a,b)$ be an arc passing through the pair of crossings ($c_{1}$,$c_{2}$). Without loss of generality assume that $c_{1}$ is classical crossing while $c_{2}$ being virtual. By arc shift move on the arc $(a,b)$, we mean cutting the arc at two points near $a$ and $b$ and identifying the loose ends on one side with loose ends on the other side in the way as shown in Fig.~\ref{ARCGLUE}. While applying the arc shift move some new crossings may arise in the diagram, we label them as virtual crossings.

\begin{figure}[H]
\centering
  \includegraphics[width=10.6cm,height=1.35cm]{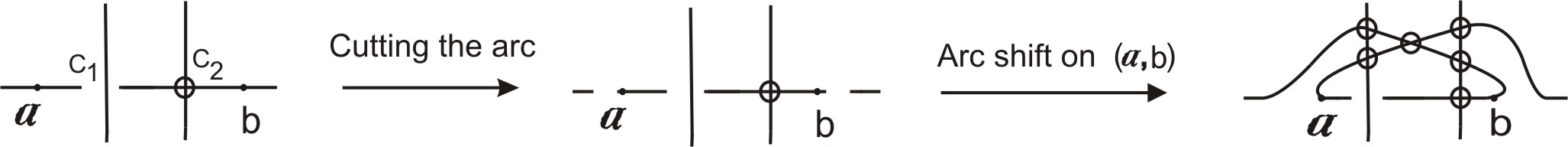}
 \caption{  Arc shift on arc $(a,b)$  }
  \label{ARCGLUE}
\end{figure}

\begin{remark}
There can be many possible ways to join the loose ends in the diagram while applying arc shift move. Therefore, there are number of diagrams corresponding to arc shift move on the arc $(a,b)$ in $D$, two such diagrams are shown in Fig.~\ref{REMARK1}. However, in all such diagrams the strands joining the loose ends contains only virtual crossings.

\begin{figure}[H]
\centering
  \includegraphics[width=8cm,height=3.2cm]{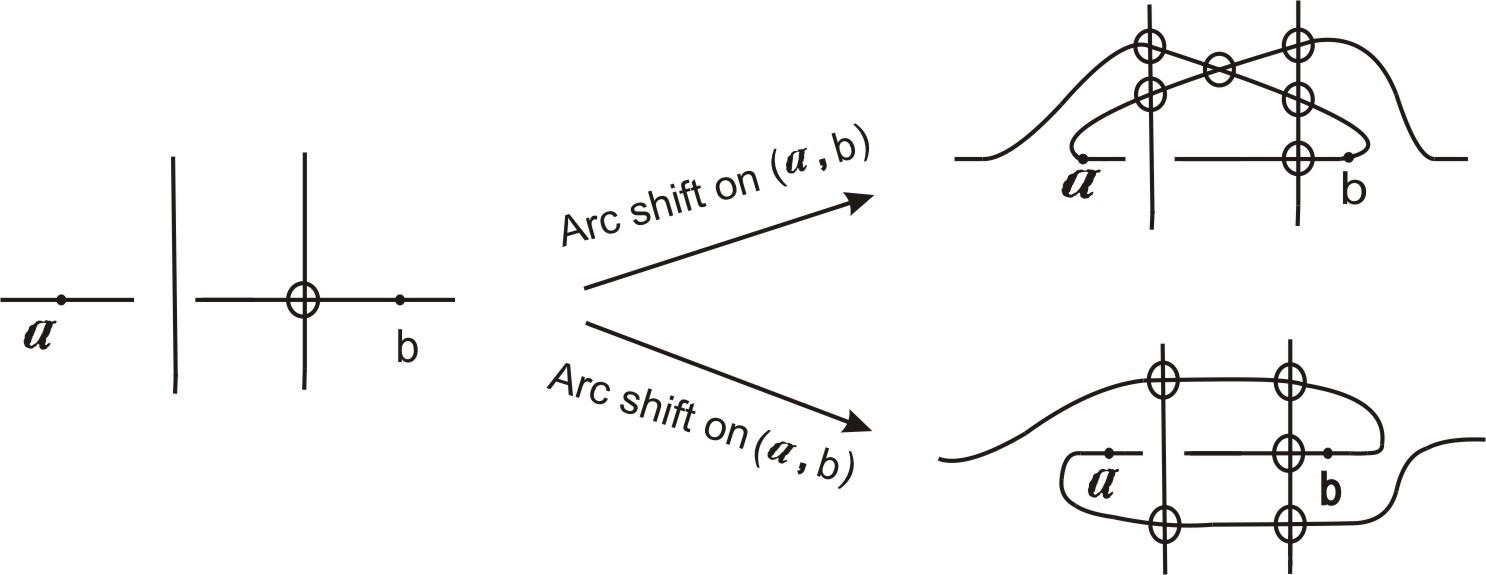}
 \caption{Equivalent diagrams corresponding to arc shift on arc $(a,b)$}  
 \label{REMARK1}
\end{figure}

Therefore, as a result of detour move (Fig.~\ref{DETOUR}) any two such diagrams are equivalent by virtual Reidemeister moves. Considering equivalence of all these diagrams, we denote the diagram obtained from $D$ as a result of arc shift on the arc $(a,b)$  by $D_{(a,b)}$.
\end{remark}

\begin{remark}
As an effect of arc shift move on the arc $(a,b)$ in an oriented virtual knot diagram $D$, orientation in the encircled region  gets reversed as shown in Fig.~\ref{AKSORREVERSE}.
\end{remark}

\begin{figure}[H]
\centering
  \includegraphics[width=9.7cm,height=2.4cm]{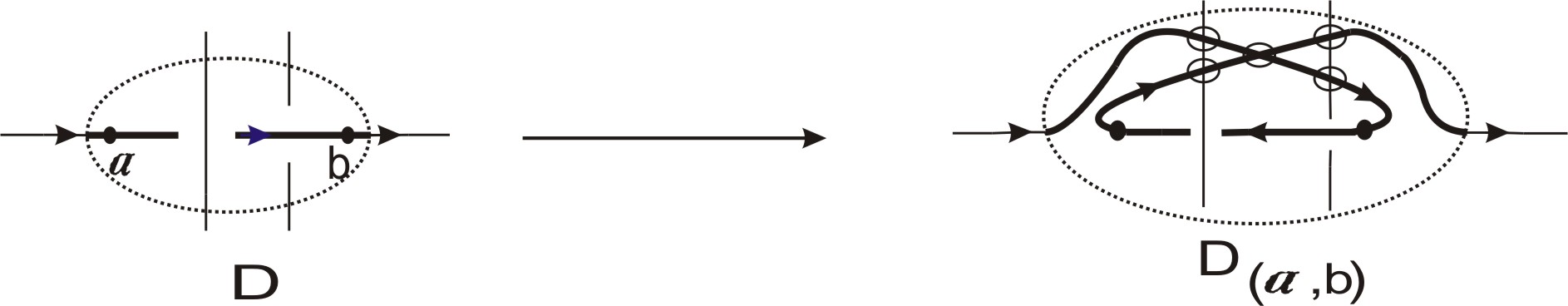}
 \caption{  }
  \label{AKSORREVERSE}
\end{figure}
Depending on the crossing type (classical/virtual) and crossing information (over/under), an arc can have different possible local configurations. However, five of these local configurations are enough to summarize effects of all possible arc shift moves in Gauss diagram. Corresponding to these five cases, we denote the respective arc shift moves by $ \boldsymbol{\bar{A}_h} $, $ \boldsymbol{\bar{A}_t} $, $ \boldsymbol{\bar{A}_{ht}} $,$ \boldsymbol{\bar{A}_{th}} $ and $ \boldsymbol{\bar{A}_s} $(see Fig.~\ref{TYPEAKS2}).

\begin{figure}[H]
\centering
  \includegraphics[width=10cm,height=8cm]{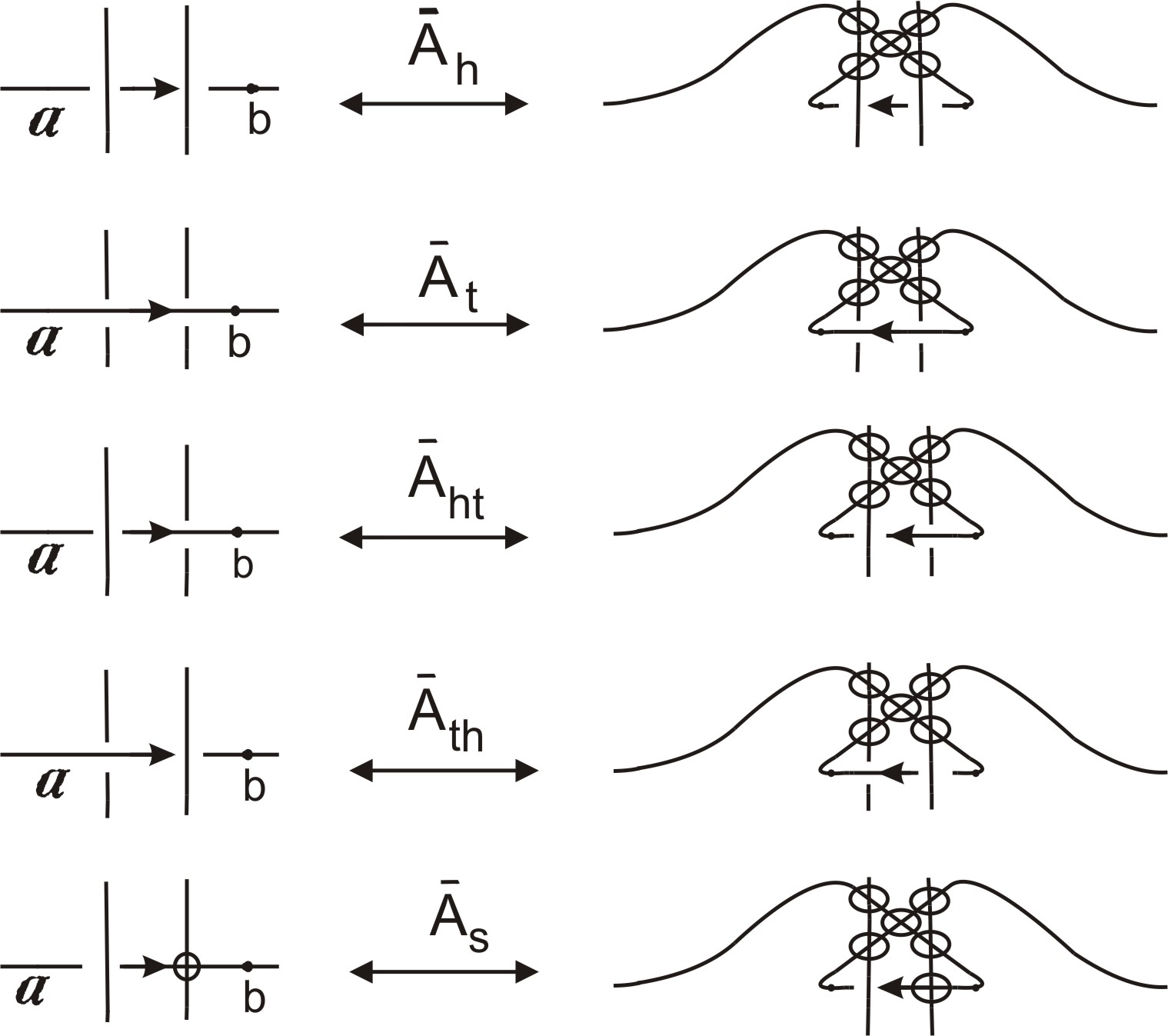}
 \caption{ arc shift move on arc $(a,b)$}
 \label{TYPEAKS2}
\end{figure}
\noindent Moves corresponding to  $ \boldsymbol{\bar{A}_h} $, $ \boldsymbol{\bar{A}_t} $, $ \boldsymbol{\bar{A}_{ht}} $,$ \boldsymbol{\bar{A}_{th}} $ and $ \boldsymbol{\bar{A}_s} $ in Gauss diagram are shown in Fig.~\ref{ARCSHIFTGD1}.

\begin{figure}[H]
\centering
  \includegraphics[width=8cm,height=9cm]{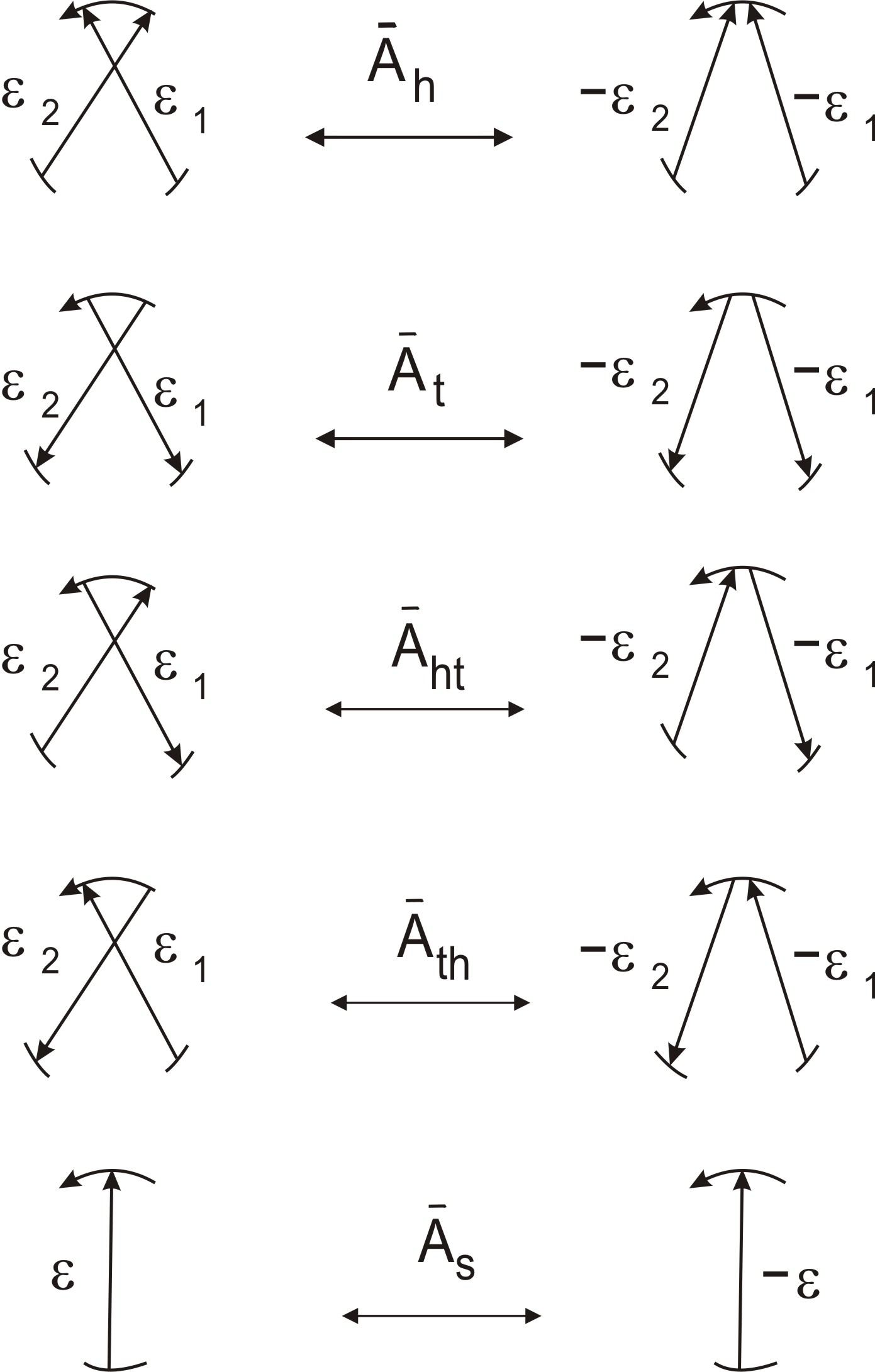}
 \caption{Gauss diagram analogues to the arc shift moves shown in Fig.~\ref{TYPEAKS2}}
  \label{ARCSHIFTGD1}
\end{figure}

\noindent Adjacent ends of two arrows in a Gauss diagram $G(D)$ may not always correspond to an arc in $D$.  There might be number of consecutive virtual crossings in $D$ between the crossings $c_{1}$, $c_{2}$ corresponding to adjacent pair of ends of the arrows. However, $D$ can be altered to an equivalent diagram $D'$ having an arc containing ($c_{1}$,$c_{2}$) as explained in the following proposition.
\begin{proposition} Let $D$ be a virtual knot diagram and ($c_{1}$,$c_{2}$) be a pair of classical crossings in $D$ such that the segment between $c_1$ and $c_2$ contains only virtual crossings. Then, there exists an virtual knot diagram $D'$ equivalent to $D$ where crossings ($c_{1}$,$c_{2}$) are contained in an arc.
\end{proposition}
\begin{proof}
Consider the diagram $D$ and apply Detour move on each of the vertical segments containing virtual crossings as shown in Fig.~\ref{PROP1}. After applying Detour move finite number of times in $D$, the segment between $c_1$ and $c_2$ becomes free of virtual crossings. Diagram $D'$ obtained as a result is equivalent to $D$ and arc $(a,b)$ in $D'$ contains crossings ($c_{1}$,$c_{2}$) as required in the proposition.
\end{proof}

\begin{figure}[H]
\centering
  \includegraphics[width=10cm,height=3cm]{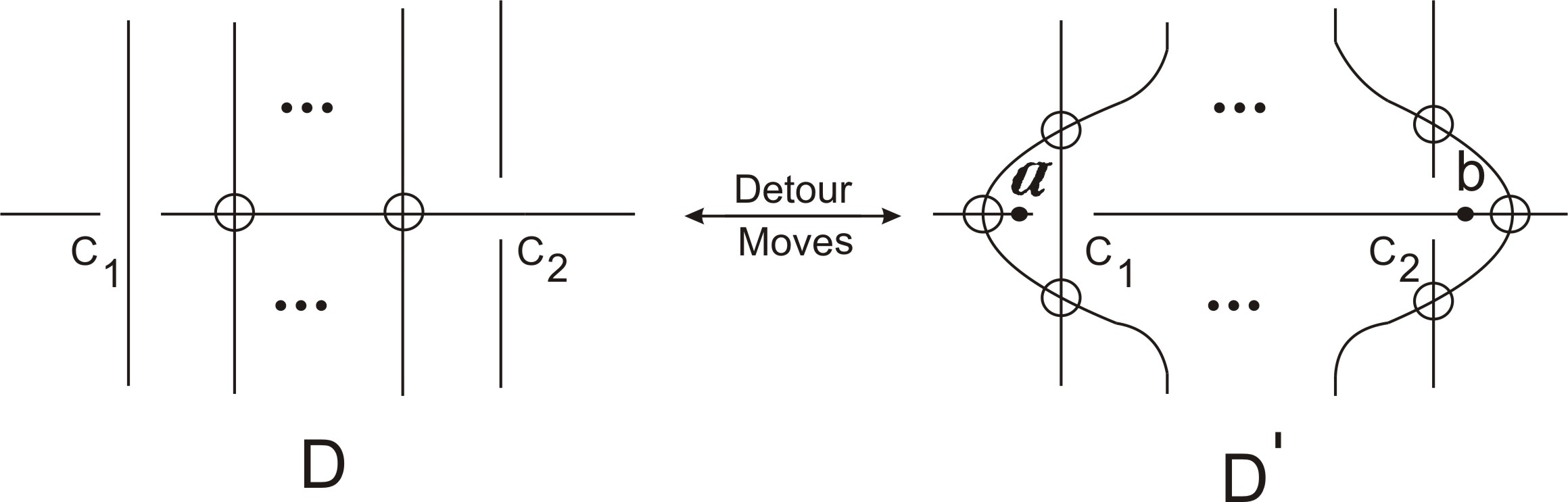}
 \caption{}
  \label{PROP1}
\end{figure}

\noindent As a consequence of proposition 1, we have following theorem, which tells that to check whether two virtual knot diagrams are related by arc shift moves and virtual Reidemeister moves, it is enough to check the equivalence of corresponding Gauss diagrams by moves given in Fig.~\ref{ARCSHIFTGD1}. 
\begin{theorem}
Let $G(D_1)$ and $G(D_2)$ be Gauss diagrams related by a finite sequence of diagrammatic moves $ \boldsymbol{\bar{A}_h} $,$ \boldsymbol{\bar{A}_t} $,$ \boldsymbol{\bar{A}_{ht}} $ and $ \boldsymbol{\bar{A}_{th}}$ given in Fig.~\ref{ARCSHIFTGD1}. Then, corresponding virtual knot diagrams $D_1$ and $D_2$ can be obtained from each other by respective arc shift moves and virtual Reidemeister moves.
\end{theorem}
\begin{proof}
Choose first move in the sequence relating $G(D_1)$ and $G(D_2)$. Consider the classical crossings $c_1$ and $c_2$ in $D_1$ corresponding to pair of adjacent arrows affected by first move in Gauss diagram $G(D_1)$. Proposition 1 guarantees existence of a diagram $D'$ equivalent to $D_1$ by virtual Reidemeister moves such that crossings ($c_{1}$,$c_{2}$) are contained in an arc $(a,b)$ in $D'$. Both $D_1$ and $D'$ being equivalent by virtual Reidemeister moves have identical Gauss diagrams, therefore, arc shift move on the arc $(a,b)$ in $D'$ results in first move chosen from the sequence. Virtual knot diagram so obtained from $D'$ is related to $D_1$ by one arc shift move and virtual Reidemeister moves. Similarly the process continues and for the last move in the sequence we get the virtual knot diagram $D_2$ corresponding to $G(D_2)$ and related to $D_1$ via arc shift moves and virtual Reidemeister moves.
\end{proof}

\noindent Virtual knot diagram $D_{(a,b)}$ contains an arc containing same pair of crossings ($c_{1}$,$c_{2}$) as contained by arc $(a,b)$ in $D$. Applying arc shift move again on the corresponding arc in $D_{(a,b)}$ results in a diagram equivalent to original diagram $D$ as we discuss in the following proposition.
\begin{proposition}
Let $D$ be a virtual knot diagram and $D'$ is obtained from $D$ by applying arc shift move twice on an arc $(a,b)$. For the resulting diagram we have $D'\sim D$.
\end{proposition}
\begin{proof}
Let $(a,b)$ be an arc in $D$ passing through the pair of crossings ($c_{1}$,$c_{2}$). For convenience assume that both the crossings $(c_1,c_2)$ are classical having crossing information as shown in Fig.~\ref{PROPOS2}(1). We obtain Fig.~\ref{PROPOS2}(2) by applying arc shift on the arc $(a,b)$ in Fig.~\ref{PROPOS2}(1). Again, applying arc shift in Fig.~\ref{PROPOS2}(2) results in the diagram $D'$ shown in Fig.~\ref{PROPOS2}(3) where if we apply $VR_2$ move in each of the two encircled regions we get Fig.~\ref{PROPOS2}(4). In Fig.~\ref{PROPOS2}(4), if we apply three $VR_2$ moves in the encircled region, we obtain Fig.~\ref{PROPOS2}(5) which is identical to diagram $D$ we started with, i.e , Fig.~\ref{PROPOS2}(1). Thus, $D'$ is equivalent to $D$ as required.\\
Similarly all the other cases involving $(c_1,c_2)$ having different crossing type (classical/virtual) and crossing information (over/under) follows.
\end{proof} 
\begin{figure}[H]
\centering
  \includegraphics[width=11cm,height=6cm]{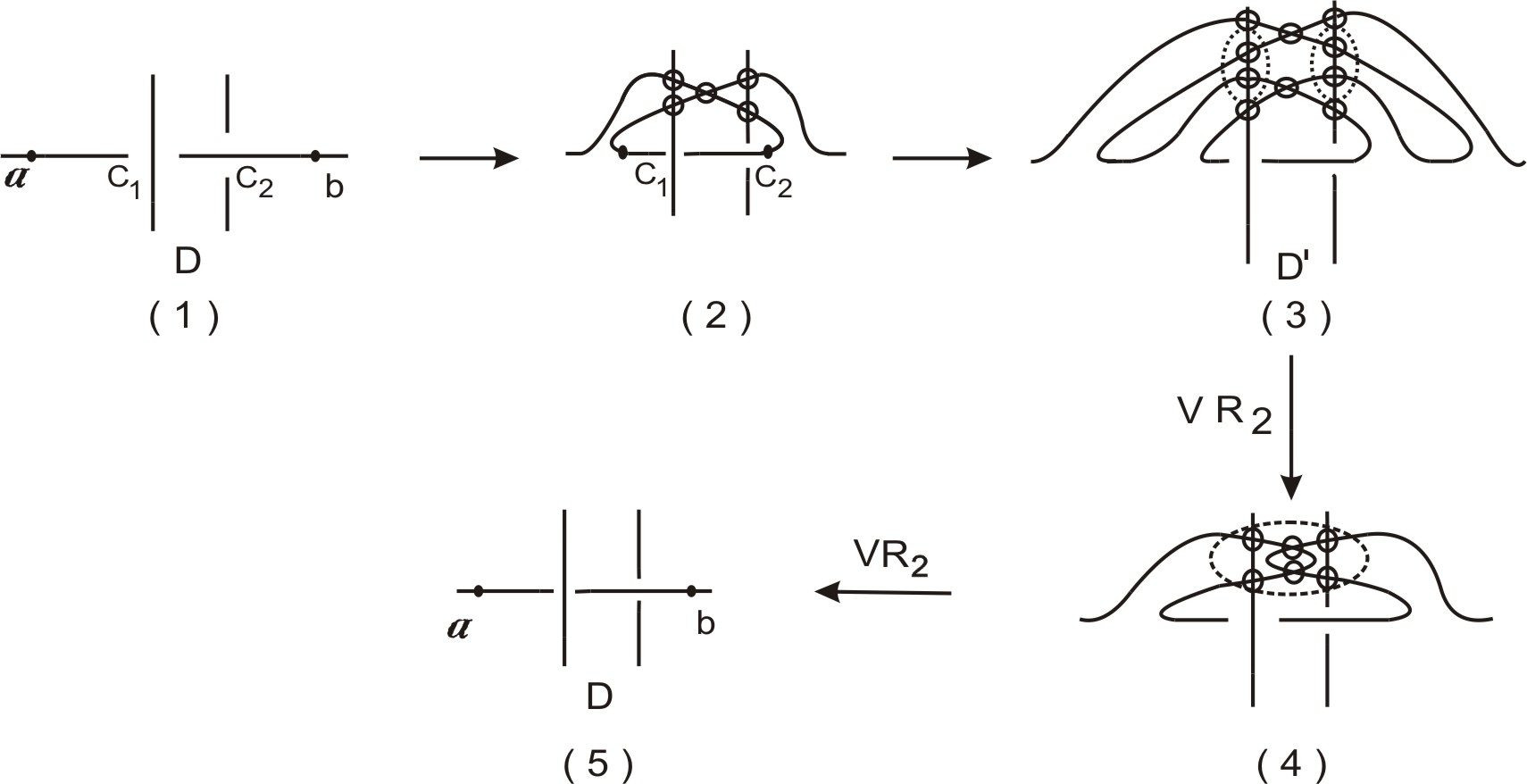}
 \caption{}
  \label{PROPOS2}
\end{figure}

\noindent By applying single arc shift move in equivalent diagram of an oriented virtual knot diagram we can realize switch in the sign of any crossing $c$(without changing the crossing information). While doing this, all the other crossings remain unaffected as discussed in the following proposition.

\begin{proposition}Let $D$ be a virtual knot diagram and $c$ be any crossing in $D$. Then, there exists a diagram $D'$ obtained from $D$ by applying an arc shift move such that the crossing $c'$ in $D'$ corresponding to $c$ is of opposite sign, i.e., sign($c'$)= \begin{large}
-
\end{large} sign($c$).
\end{proposition}

\begin{proof}
Consider any crossing $c$ in $D$ as shown in Fig.~\ref{SIGNCHANGE}(1). Now, first apply $VR_1$ move in Fig.~\ref{SIGNCHANGE}(1) that results in Fig.~\ref{SIGNCHANGE}(2) which has an arc $(b,d)$ containing crossing $c$ and a virtual crossing. Apply an arc shift move on the arc $(b,d)$ in Fig.~\ref{SIGNCHANGE}(2) to obtain Fig.~\ref{SIGNCHANGE}(3) where the segment from $b$ to $e$ contains only virtual crossings. Using Detour move in Fig.~\ref{SIGNCHANGE}(3) we get diagram $D'$ in Fig.~\ref{SIGNCHANGE}(4) where sign($c'$)= - sign($c$) as required. Sign of other crossings remains unchanged as all the other crossings in $D$ and $D'$ have same local orientation.
\end{proof} 
\begin{figure}[H]
\centering
  \includegraphics[width=10cm,height=5.5cm]{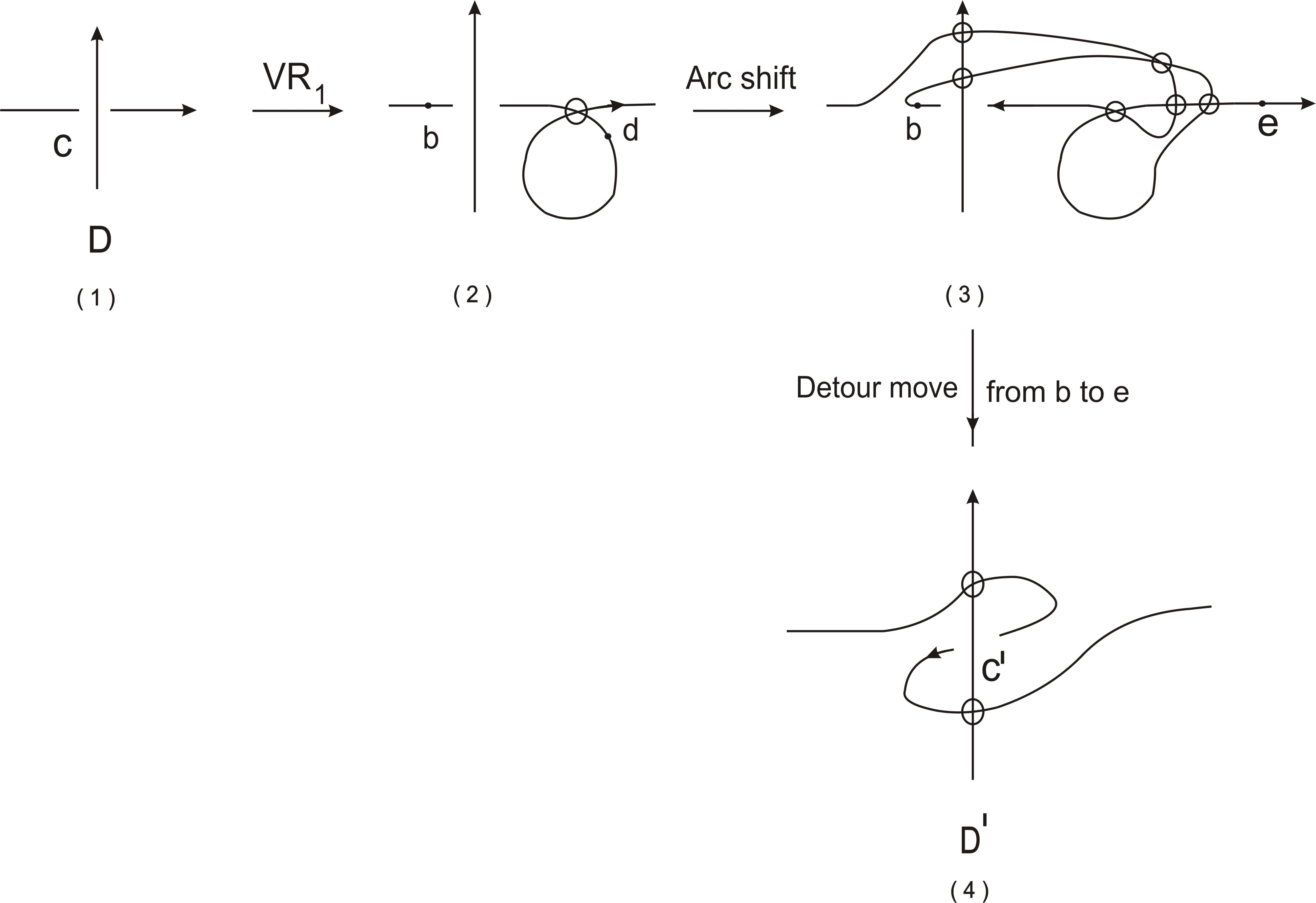}
 \caption{sign($c'$)= - sign($c$)}
  \label{SIGNCHANGE}
\end{figure}

\begin{proposition}Let $D$ and $D'$ be two virtual knot diagrams that differ by a $R_3$ move, then $D'$ can be obtained from $D$ by applying three arc shift moves.
\end{proposition}
\begin{proof} Consider the Gauss diagrams of $D$ and $D'$ related by a $R_3$ move(see Fig.~\ref{ARCSHIFTR31}). Arc shift moves $\boldsymbol{\bar{A}_{th}}$, $\boldsymbol{\bar{A}_{h}}$ and $\boldsymbol{\bar{A}_{t}}$ applied in succession realize same changes in the Gauss diagram as done by a $R_3$ move as shown in the Fig.~\ref{ARCSHIFTR31}. By theorem 3.1, virtual knot diagrams $D$ and $D'$ corresponding to the Gauss diagrams $G(D)$ and $G(D')$ are related by arc shift moves and virtual Reidemeister moves. Therefore, three arc shift moves together with some virtual Reidemeister moves are enough to realize a $R_3$ move in $D$.
\end{proof}
\begin{figure}[H]
\centering
  \includegraphics[width=8cm,height=7cm]{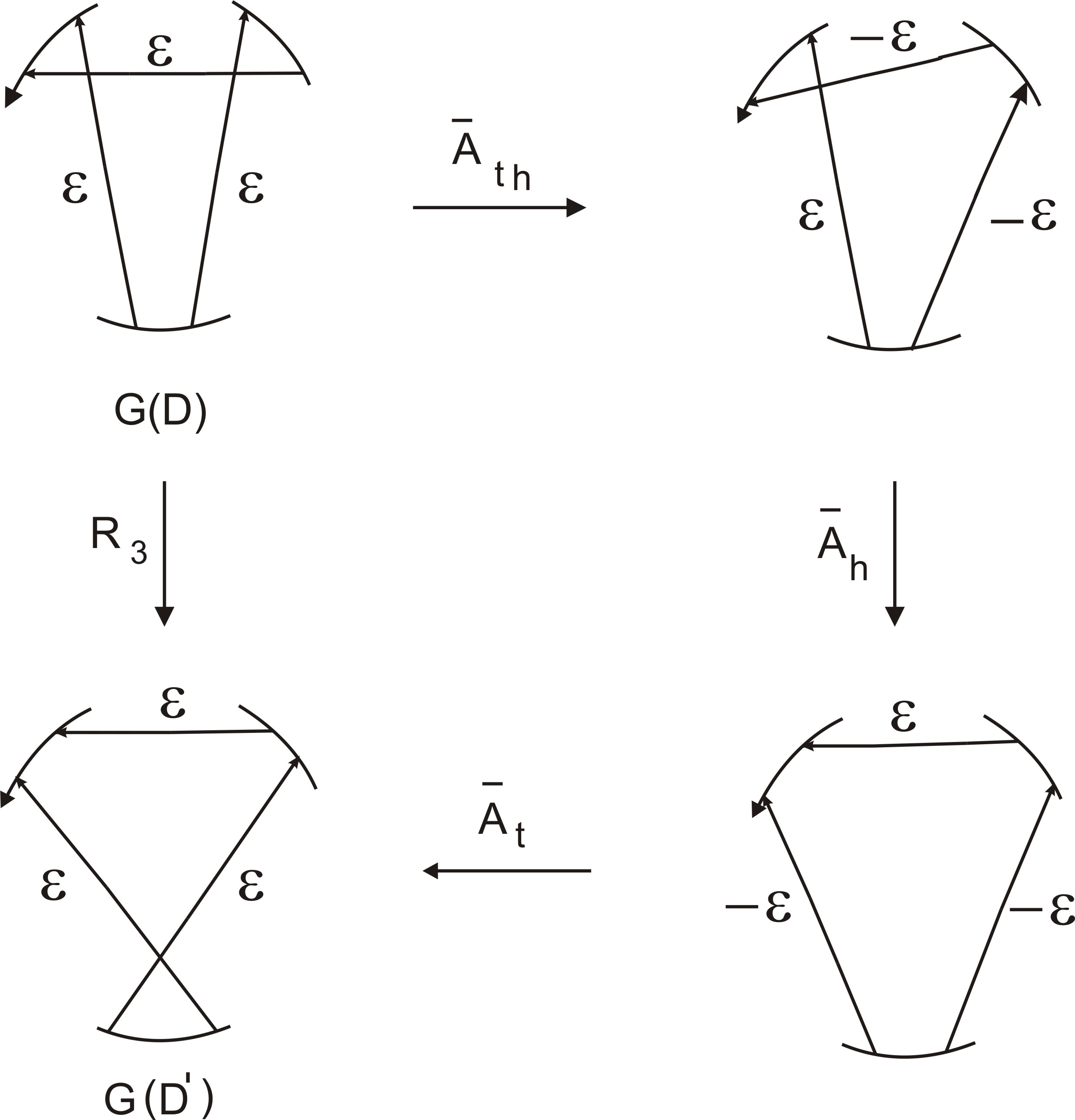}
 \caption{$R_3$ move realized via arc shift moves }
  \label{ARCSHIFTR31}
\end{figure}

\noindent H. Murakami and Y. Nakanishi \cite{murakami1989certain} defined the $\Delta$-move as shown in Fig.~\ref{DELTAMOVE} and established that a classical knot can be unknotted using the $\Delta$-move.

\begin{figure}[H]
\centering
  \includegraphics[width=8cm,height=2cm]{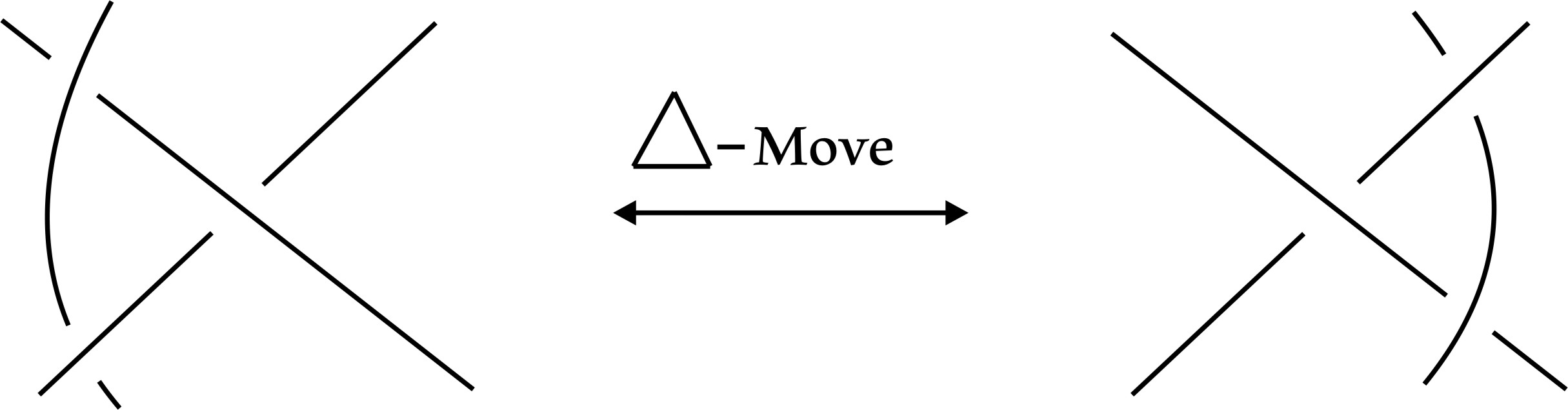}
 \caption{ }
  \label{DELTAMOVE}
\end{figure}
\noindent In the folllowing lemma we show that, a $\Delta$-move can be realized by arc shift moves and virtual Reidemeister moves.
\begin{lemma}Given a virtual knot diagram $D$, let $D'$ is obtained from $D$ by applying a $\Delta$-move. Then, there exists arc shift moves which applied in $D$ gives an equivalent diagram of $D'$.
\end{lemma}
\begin{proof}
Consider the Gauss diagram $G(D)$ corresponding to virtual knot diagram $D$ as shown in Fig.~\ref{DELTAGD1} and apply $\Delta$-move to get $G(D')$. Now, applying three arc shift moves $\boldsymbol{\bar{A}_{ht}}$, $\boldsymbol{\bar{A}_{th}}$ and $\boldsymbol{\bar{A}_{ht}}$ in sequence realizes same change in $G(D)$ as by a $\Delta$-move. By theorem 3.1, Gauss diagrams $G(D)$ and $G(D')$ correspond to virtual knot diagrams related by arc shift moves and virtual Reidemeister moves and hence the result follows.
\end{proof}   
\begin{figure}[H]
\centering
  \includegraphics[width=7cm,height=6cm]{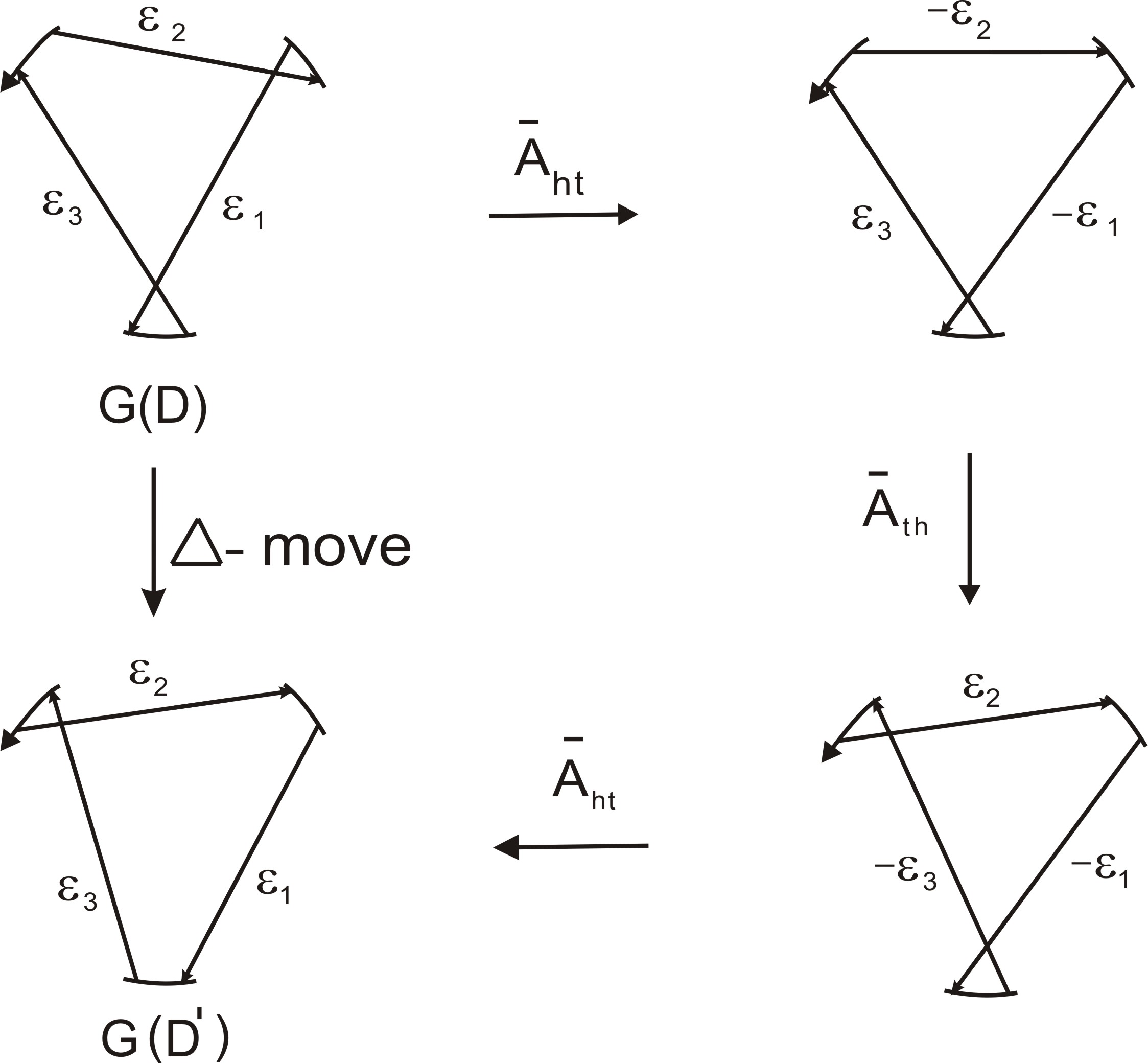}
 \caption{ $\Delta$-move realized using arc shift moves }
  \label{DELTAGD1}
\end{figure}

\noindent As $\Delta$-move is an unknotting operation for classical knots, Lemma 1 guarantees that any classical knot diagram can be transformed into trivial knot diagram using arc shift moves and virtual Reidemeister moves.
\section{Arc shift as an unknotting operation for virtual knots}
Lemma 1 ensures that classical knots considered as a subclass of virtual knots can be unknotted using arc shift moves and virtual Reidemeister moves. However, the result generalizes to every virtual knot as we prove in this section. We use Gauss diagrams to prove the result using the fact that a Gauss diagram defines virtual knot uniquely upto equivalence by moves in the Fig.~\ref{GDRMOVES1}. A Gauss diagram in which no two arrows intersect is called parallel chord diagram and corresponds to a trivial knot.
\begin{theorem} Every virtual knot diagram $D$ can be transformed into trivial knot diagram using arc shift moves and generalized Reidemeister moves.
\end{theorem}
\begin{proof}
It is enough to prove that using arc shift moves and generalized Reidemeister moves in $D$, Gauss diagram $G(D)$ corresponding to $D$ can be turned into a parallel chord diagram. With anticlockwise orientation on $G(D)$ choose a random arrow and consider the next arrow adjacent to the head of chosen arrow along orientation. Crossings in $D$ corresponding to the two arrows may have only virtual crossings between them. By proposition 1, there exists an equivalent diagram $D'$ of $D$ where both the crossings are contained in an arc $(a,b)$. Arc shift on the arc $(a,b)$ in $D'$ moves across head of the chosen arrow with adjacent arrow in $G(D)$ and also switches signs of both arrows. Continue the process for all the arrows encountered with head of the chosen arrow along orientation till we reach a Gauss diagram having no arrow between head and tail of the random arrow we started with. In the process, we used virtual Reidemeister moves and arc shift moves in $D$ to realize change in Gauss diagram $G(D)$ that makes a randomly chosen arrow free of intersections by other arrows. Repeating the process for all the arrows one by one gives us a Gauss diagram where no two arrows intersect each other, i.e.,a parallel chord diagram as required. Sign of the some of the arrows might change in the whole process but has no affect on the final result as any parallel chord diagram irrespective of the signs of the chords corresponds to trivial knot. Fig.~\ref{ISOLATE} shows an example of turning a Gauss diagram into parallel chord diagram.

\end{proof}


\begin{figure}[H]
\centering
  \includegraphics[width=10cm,height=2cm]{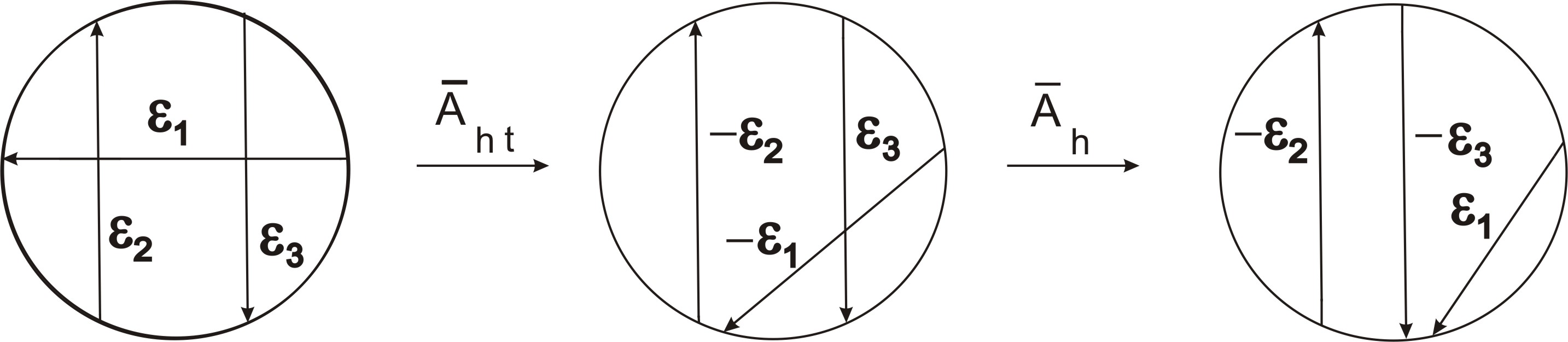}
 \caption{Turning a Gauss diagram into parallel chord diagram }
  \label{ISOLATE}
\end{figure}

\noindent we give an example of theorem 4.1.

\begin{example}Consider the virtual knot diagram $D$ of left handed virtual trefoil knot(Fig.~\ref{UNKNOTVT}). $D$ is transformed into trivial knot using single arc shift move on the arc $(a,b)$ followed by generalized Reidemeister moves.
\end{example}
\begin{figure}[H]
\centering
  \includegraphics[width=11cm,height=4cm]{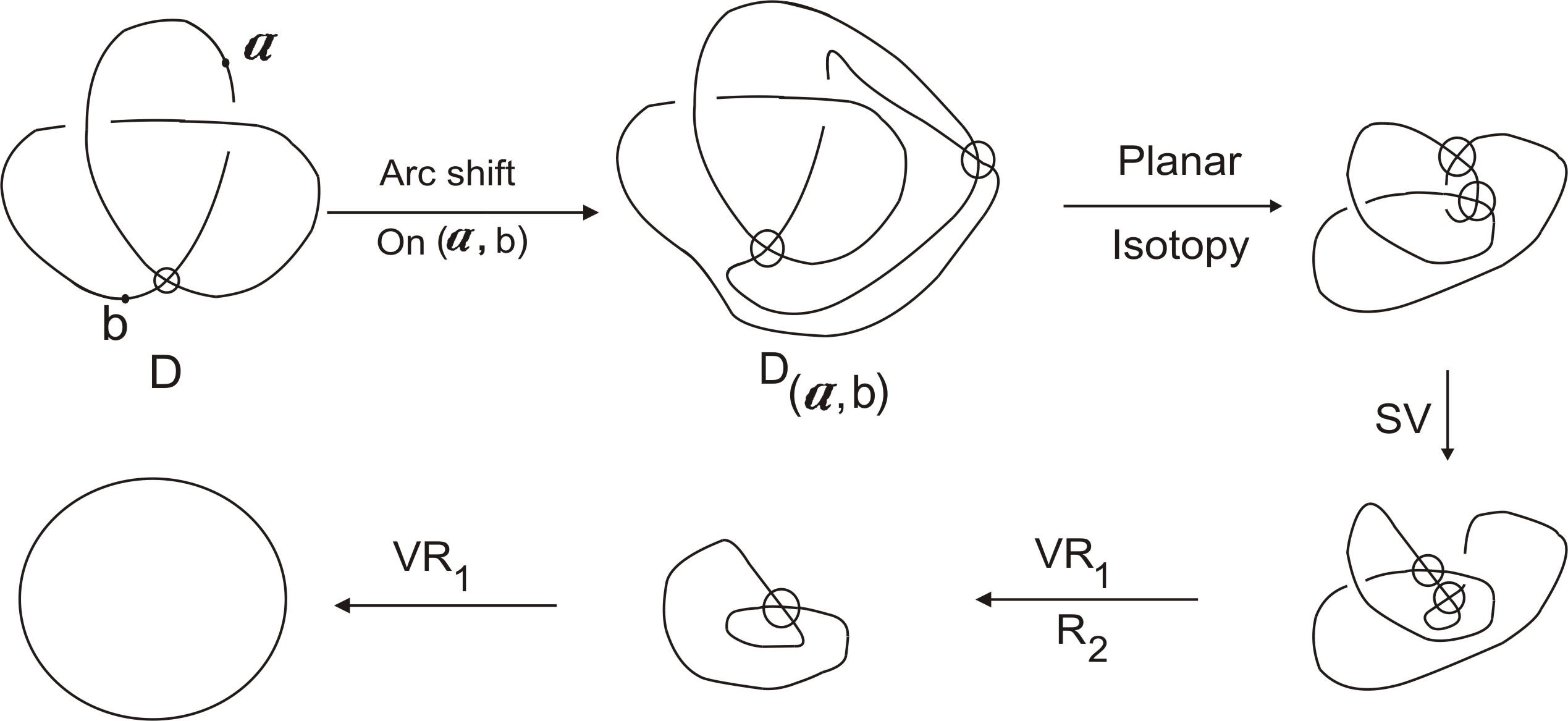}
 \caption{Unknotting virtual trefoil using arc shift move}
  \label{UNKNOTVT}
\end{figure}
\begin{proposition}Let $D$ and $D'$ be two virtual knot diagrams. Let $n$ and $m$ be the minimum number of arc shift moves needed to transform $D$ and $D'$ respectively into trivial knot. If $D \sim D'$ then $n=m$.
\end{proposition}
\begin{proof}
Since both $D$ and $D'$ are equivalent, there exists a sequence of generalized Reidemeister moves relating $D$ with $D'$. Suppose $D$ can be turned into trivial knot using $n$ number of arc shift moves and some generalized Reidemeister moves. Using sequence of generalized Reidemeister moves relating $D$ with $D'$ we first transform $D'$ into $D$ and then use $n$ arc shift moves to turn $D$ into trivial knot. Similarly if $m$ number of arc shift moves are needed to turn $D'$ into trivial knot then $D$ can also be made trivial using $m$ arc shift moves. Taking minimum over all such $m$ and $n$ gives desired result.
\end{proof}

\noindent This motivates us to define the \emph{arc shift number} for a virtual knot.
\begin{definition} \emph{For any virtual knot $K$, the} arc shift number \emph{of $K$, $A(K)$ is the minimum number of arc shift moves needed to turn a diagram of $K$ into trivial knot.}
\end{definition}

Since any diagram of trivial knot can be converted into unknot using generalized Reidemeister moves, no arc shift move is needed and hence $A(K)$ is zero for trivial knot. However, a diagram of nontrivial knot necessarily needs arc shift moves to be converted into trivial knot. Therefore $A(K)$ is strictly positive for a nontrivial classical or virtual knot.

L.H. Kauffman \cite{kauffman2004self} defines parity of a crossing $c$ of a virtual knot diagram $K$. The parity of $c$ is odd if odd number of classical crossings are encountered while moving along the diagram on any path that starts and ends both at $c$, otherwise the crossing is called even. Likewise, a crossing $c$ is odd(even) if and only if the chord corresponding to $c$ in Gauss diagram intersects odd(even) number of chords. Chords corresponding to odd(even) crossings are referred as odd(even) chords respectively. Denote by \emph{Odd(K)} the set containing all the odd crossings in $K$, then sum of the signs of all the crossings in \emph{Odd(K)} is called \emph{odd writhe} of $K$ denoted by $J(K)$, i.e.,
\hspace*{4cm} $J(K) = \sum\nolimits_{c\in Odd(K)}sign($c$)$.

\noindent Equivalently, $J(K)$ is the sum of the signs of odd chords in the Gauss diagram corresponding to $K$. $J(K)$ is a virtual knot invariant and is zero for classical knots. As a consequence, whenever $J(K)$ is nonzero $K$ is necessarily nonclassical. We give a lowerbound on $A(K)$ in terms of odd writhe $J(K)$ by analyzing the change occurring in $J(K)$ for two virtual knot diagrams that differ by a single arc shift move.
\begin{proposition} If $D$ and $D'$ are two virtual knot diagrams that differ by an arc shift move, then either $J(D') = J(D)$ or $J(D') = J(D) \pm 2$.
\end{proposition}
\begin{proof}
Let the arc shift move be one of $\boldsymbol{\bar{A}_h} $, $ \boldsymbol{\bar{A}_t} $, $ \boldsymbol{\bar{A}_{ht}} $ or $ \boldsymbol{\bar{A}_{th}} $.
Consider the Gauss diagrams corresponding to $D$ and $D'$. We note that arc shift moves $\boldsymbol{\bar{A}_h} $, $ \boldsymbol{\bar{A}_t} $, $ \boldsymbol{\bar{A}_{ht}} $ and $ \boldsymbol{\bar{A}_{th}} $ moves adjacent ends of two chords past each other switching both the signs. As a result, parity of the two chords involved gets flipped(odd/even to even/odd) and remaining chords maintains the same parity. With a slight abuse of notation we denote both crossings and chords corresponding to them by $c_1$,$c_2$ and assume that $sign(c_1)= \varepsilon_1$, $sign(c_2)= \varepsilon_2$. Let $c'_1$,$c'_2$ denotes the corresponding chords after applying the arc shift move, thus $sign(c'_1)= -\varepsilon_1$, $sign(c'_2)= -\varepsilon_2$ and $c'_1$,$c'_2$ have parity opposite to $c_1$,$c_2$ respectively. We discuss all four cases based on the parity of $c_1$,$c_2$ and note the corresponding change in odd writhe $J(D)$.

\noindent \textbf{Case 1}: When both $c_1$,$c_2$ are even.\\
 $c_1$,$c_2$ being both even do not contribute to $J(D)$, while $c'_1$,$c'_2$ both being odd contribute in $J(D')$. We have
\begin{equation}
  \begin{aligned}
         J(D') & = J(D) + sign(c'_1)+sign(c'_2) \\
               & = J(D) + (-\varepsilon_1)+(-\varepsilon_2)\\
               & = J(D) - (\varepsilon_1+\varepsilon_2),
 \end{aligned}
\end{equation}

\noindent \textbf{Case 2}: When both $c_1$,$c_2$ are odd.\\
 $c_1$,$c_2$ being both odd contribute to $J(D)$, while $c'_1$,$c'_2$ both being even do not contributein $J(D')$. We have
\begin{equation}
  \begin{aligned}
         J(D') & = J(D) - sign(c_1)-sign(c_2) \\
               & = J(D) -\varepsilon_1-\varepsilon_2\\
               & = J(D) - (\varepsilon_1+\varepsilon_2),
 \end{aligned}
\end{equation}

\noindent \textbf{Case 3}: When $c_1$ is even and $c_2$ is odd.\\
Only $c_2$ contributes to $J(D)$, while among $c'_1$,$c'_2$ only $c'_1$ being odd contributes\\
in $J(D')$. We have
\begin{equation}
  \begin{aligned}
         J(D') & = J(D) - sign(c_2)+sign(c'_1) \\
               & = J(D) -\varepsilon_2+(-\varepsilon_1)\\
               & = J(D) - (\varepsilon_1+\varepsilon_2),
 \end{aligned}
\end{equation}

\noindent \textbf{Case 4}: When $c_1$ is odd and $c_2$ is even.\\
This case is similar to case 3.\\
Since only possible values for $\varepsilon_1+\varepsilon_2$ are $0,-2,+2 $, we have either $J(D') = J(D)$ or $J(D') = J(D) \pm 2$. 

Only remaining arc shift move $\boldsymbol{\bar{A}_s} $ switches sign of the corresponding single chord thus keeping parity of all the chords unaltered. As a result if affected chord is even then odd writhe remains same, i.e., $J(D') = J(D)$ and if affected chord is odd then we have $J(D') = J(D) \pm 2$.
\end{proof}

In the following theorem using proposition 6, we provide a lowerbound to the arc shift number $A(K)$.
\begin{theorem}For a virtual knot $K$, arc shift number $A(K)\geq |J(K)|/2$.
\end{theorem}
\begin{proof}
Let $A(K)= n$ and $K=K_0\rightarrow K_1\rightarrow K_2\rightarrow K_3\rightarrow\cdots\rightarrow K_t$ be the sequence realizing $A(K)$. $K_t$ is unknot diagram and each $K_i$ is obtained from $K_{i-1}$ by either an arc shift move or generalized Reidemeister move. Exactly $n$ terms in the sequence correspond to arc shift move to realize $A(K)$. We have,
\begin{equation}
  \begin{aligned}
         |J(K_t)-J(K_0)| & = |J(K_t)-J(K_{t-1})+ J(K_{t-1})-J(K_{t-2})+\cdots+J(K_1)-J(K_0)|\\\\
               & \leq |J(K_t)-J(K_{t-1})|+ |J(K_{t-1})-J(K_{t-2})|+\cdots+|J(K_1)-J(K_0)|.
 \end{aligned}
\end{equation}
Note that $J(K_t)=0$ and $J(K_0)=J(K)$. In the inequality (5) exactly $n$ sums correspond to arc shift moves and rest all corresponding to generalized Reidemeister moves. Using invariance of odd writhe and proposition 5, we have,
\begin{equation}
 |J(K)|\leq 2n.
\end{equation}
Thus $ n\geq|J(K)|/2$ and hence the result follows.
\end{proof}

We use Theorem 4.2 to determine arc shift number for virtual left hand trefoil knot, shown in Fig.~\ref{EXAMPLE2}.
\begin{example} \emph{Virtual left hand trefoil knot $K$ has $A(K)=1$. From the diagram $K$ (Fig.~\ref{EXAMPLE2}) it is immediate that $J(K)=-2$ as both the crossings are odd crossings. Using theorem 4.2 we have $A(K)\geq 1$ and as it was shown in the example 1, $K$ can be simplified into trivial knot using one arc shift move, hence $A(K)\leq 1$. We conclude that $1\leq A(K)\leq1$, i.e., $A(K)=1$.}
\end{example}

\begin{figure}[H]
\centering
  \includegraphics[width=2.5cm,height=2cm]{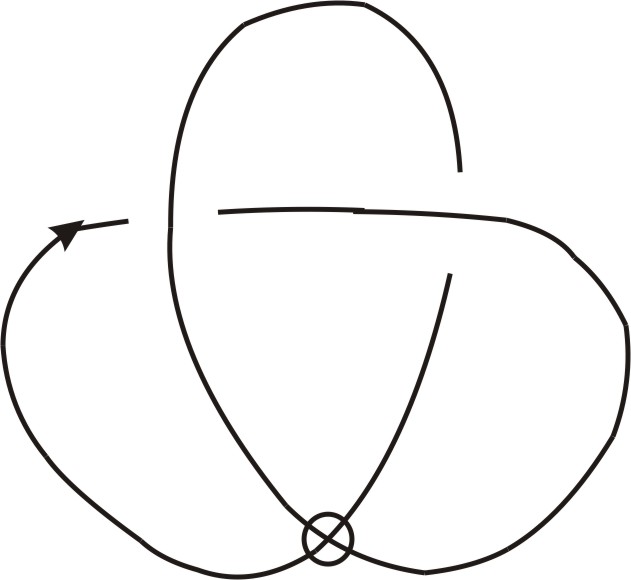}
 \caption{ }
  \label{EXAMPLE2}
\end{figure}

\section{Region arc shift}
In this section, we define {\it region arc shift operation}($RAS$) at a region in virtual knot diagram and establish it as an unknotting operation for virtual knots assisted by generalized Reidemeister moves.

\noindent For a given virtual knot diagram $D$ in $\mathbb{R}^2$, a region is a connected component of the complement of four-valent graph $D_{G}$ in $\mathbb{R}^2$, where $D_{G}$ is obtained from $D$ by replacing each classical and virtual crossings with a vertex.

\noindent {\it Region arc shift operation} at a region $R$ of diagram $D$ is a local transformation on $D$ involving arc shift operations at each arc incident on the boundary $\partial R$ of region $R$. The diagram obtained from $D$ as a result of region arc shift at the region $R$ is denoted by $D_{R}$.
\begin{figure}[H]
\centering
  \includegraphics[width=12cm,height=4cm]{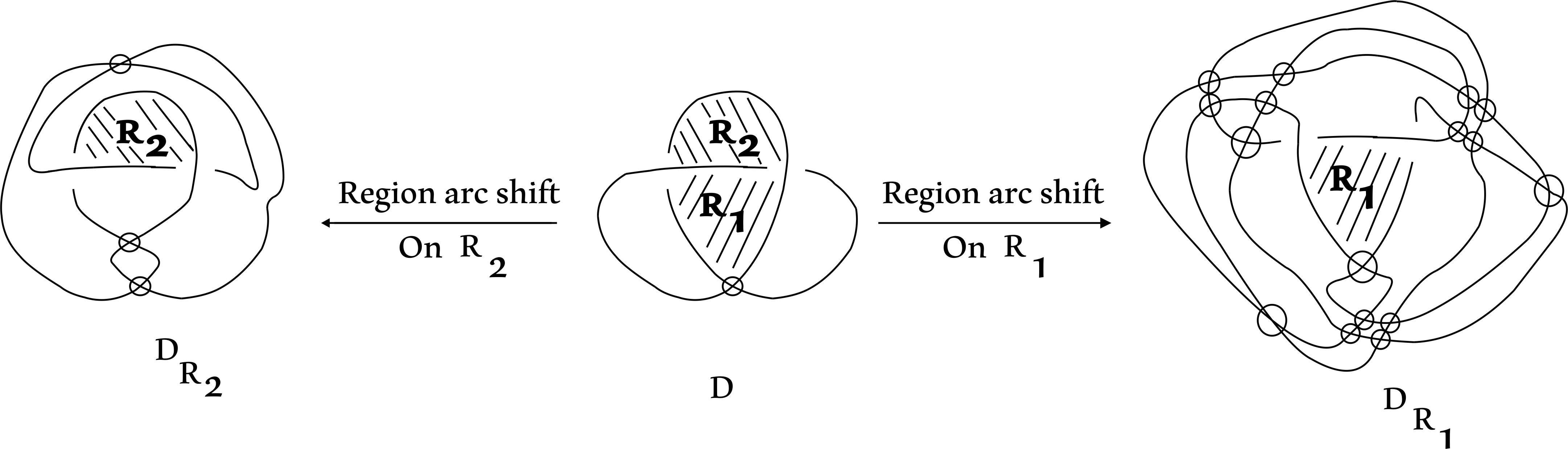}
 \caption{ Region arc shift on region $R_1$ and $R_2$}
  \label{RASVTEX1}
\end{figure}

\noindent In Fig.~\ref{RASVTEX1}, Diagrams $D_{R_1}$ and $D_{R_2}$ are obtained from diagram $D$ by applying region arc shift operation at regions $R_1$ and $R_2$ respectively. Observe that all the arcs incident on the boundary of regions $R_1$ and $R_2$ undergoes arc shift as required.
\begin{proposition}
Let $D$ be a virtual knot diagram and $R$ be a region of $D$.
If we apply the {\it region arc shift operation} consecutively two times on $R$, then resulting diagram $(D_{R})_{R}$ is equivalent to $D$.
\end{proposition}
\begin{proof}
Result follows from proposition 2.
\end{proof}
Next we prove that {\it region arc shift operation} is an unknotting operation for virtual knots along with generalized Reidemeister moves. In \cite{kanenobu2001forbidden}, T. Kanenobu proved that any virtual knot diagram can be deformed into trivial knot by applying forbidden moves and Reidemeister moves finitely many times. S. Nelson \cite{nelson2001unknotting} gave an alternative proof of the same using Gauss diagrams. Therefore, to prove that {\it region arc shift operation} is an unknotting operation, it is indeed enough to show that both forbidden moves can be realized by region arc shift operation.

\begin{proposition}
Let $D$ be a virtual knot diagram and $D'$ be the diagram obtained from $D$ by applying forbidden move $F_u$. Then, there exists a region $R$ in $D$ such that $RAS$ at region $R$ results in the diagram $D_R$ equivalent to $D'$.
\end{proposition}
\begin{proof}
Consider the diagram $D$ and a specific region $R$ in $D$ as shown in the Fig.~\ref{RASFBD1}. Boundary of region $R$ contains three arcs $\alpha, \beta$ and $\gamma$. Applying region arc shift at region $R$ results in arc shift moves on the arcs $\alpha, \beta$ and $\gamma$.
\begin{figure}[H]
\centering
  \includegraphics[width=10cm,height=4.2cm]{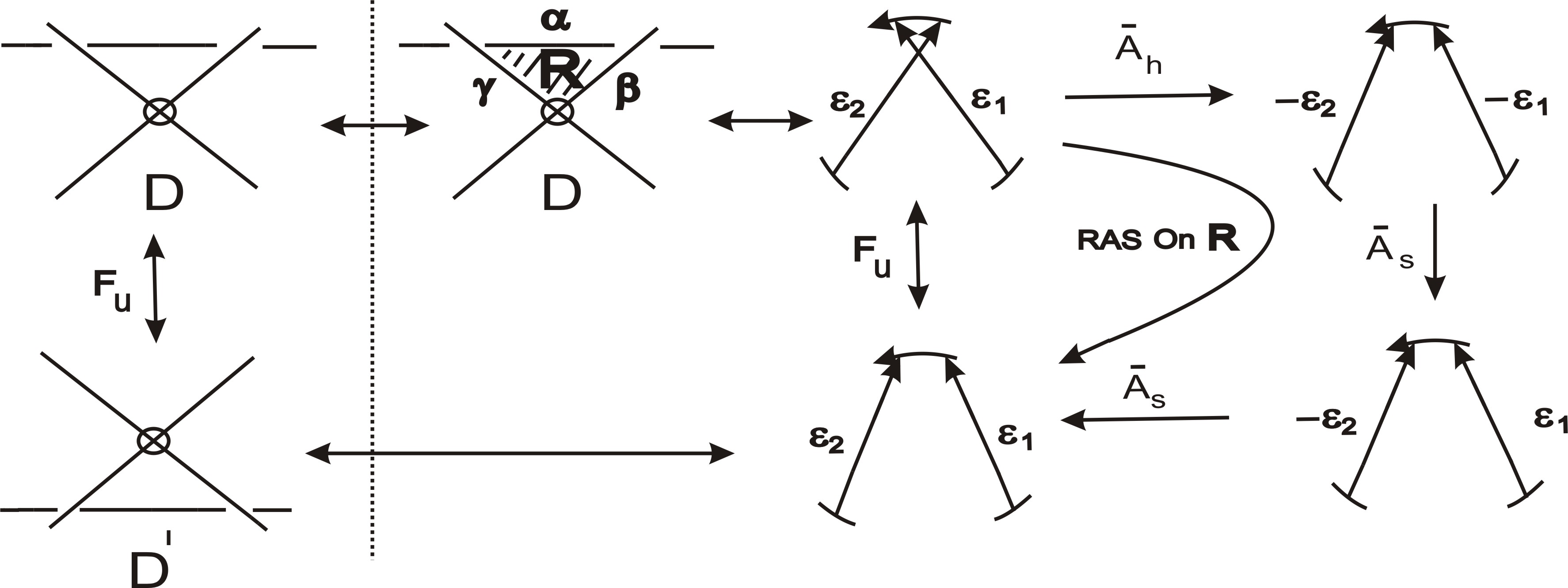}
 \caption{ Realizing forbidden move $F_{h}$ using region arc shift}
  \label{RASFBD1}
\end{figure}
It is easy to see from Fig.~\ref{RASFBD1} that while arc shift on the arc $\alpha$ corresponds to $\boldsymbol{\bar{A}_h}$, both $\beta$ and $\gamma$ correspond to $\boldsymbol{\bar{A}_s} $. We observe from Fig.~\ref{RASFBD1} that the diagram $D'$ has same Gauss diagram as has the diagram obtained by applying $RAS$ at region $R$ in $D$. Thus, $D'$ is equivalent to $D_R$ and the result follows.

Similar result can be identically proved for forbidden move $F_o$ also.
\end{proof}
\noindent As a consequence of proposition 8 we state following corollary without proof.
\begin{corollary} Every virtual knot diagram $D$ can be transformed into unknot using region arc shift operations and generalized Reidemeister moves.
\end{corollary}
\noindent On a similar note as the arc shift number we define region arc shift number as follows.
\begin{definition}
\emph{For a given virtual knot $K$, the {\it region arc shift number} $R(K)$ is the minimum number of region arc shift operations required to deform $K$ into trivial knot.}
\end{definition}

\noindent It is easy to observe that {\it region arc shift number} is a virtual knot invariant. The {\it region arc shift number} for a virtual knot $K$ is zero if and only if $K$ is trivial knot. Virtual knots shown in Fig.~\ref{RASIS1} have {\it region arc shift number} one. Region arc shift at region $R$ in both the diagrams gives a trivial knot diagram.
\begin{figure}[H]
\centering
  \includegraphics[width=8.5cm,height=3cm]{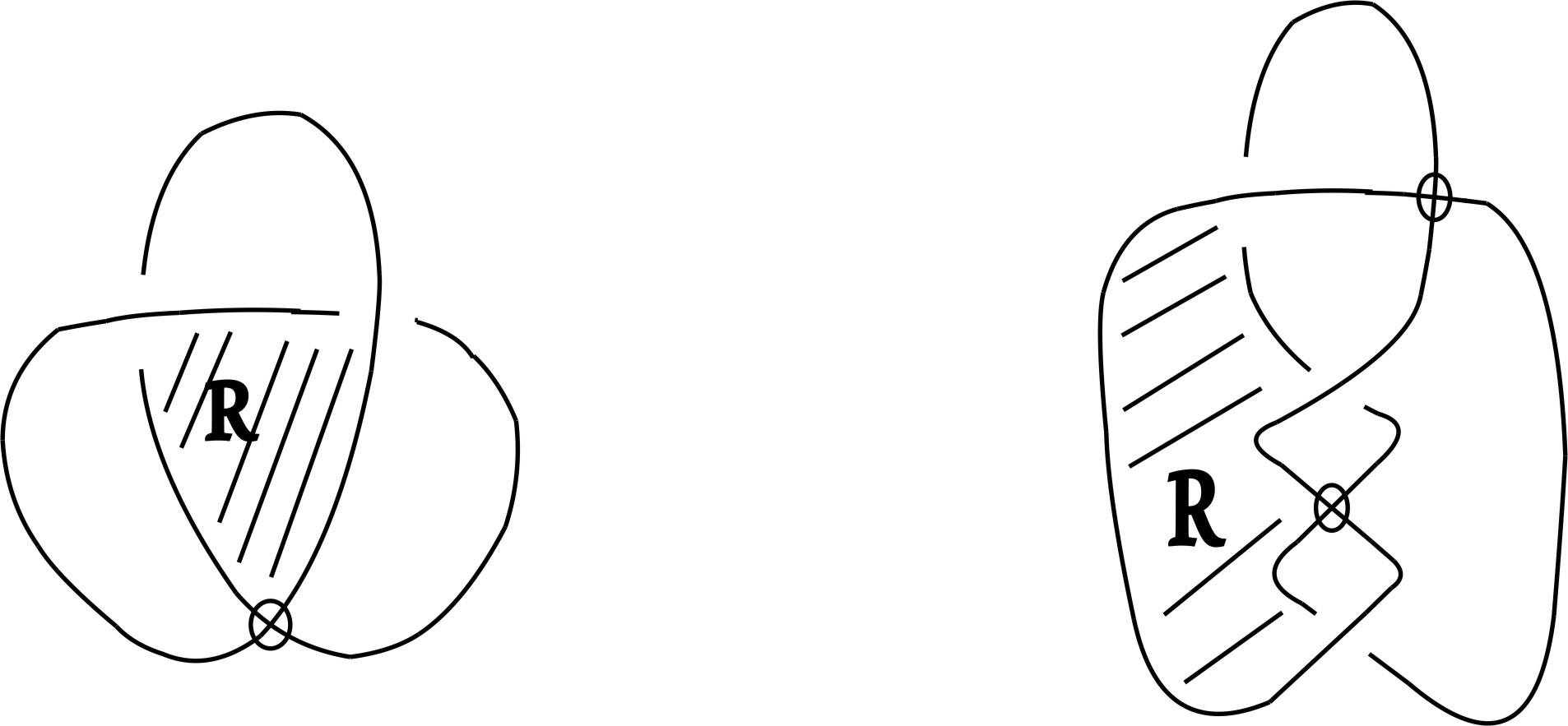}
 \caption{ Region arc shift number is $1$ for both knots }
  \label{RASIS1}
\end{figure}
\noindent In the following theorem we compare two numbers $R(K)$ and $F(K)$ and provide a relation between them in form of an inequality. 
\begin{theorem}
If $K$ is a virtual knot, then $R(K)\leq F(K).$
\end{theorem}
\begin{proof}
Suppose that forbidden number for the virtual knot $K$ is $n$. If $D$ is a diagram of $K$, then there exists a sequence involving generalized Reidemeister moves and $n$ number of forbidden moves which transforms $D$ into trivial knot diagram. Using proposition 7, we can realize each forbidden move by applying a single region arc shift operation. Replacing $n$ forbidden moves with $n$ {\it region arc shift operations} in the above sequence, we can deform $D$ into trivial knot. Thus $R(K)\leq n$ and hence the result follows.
\end{proof}

\begin{remark}
Strict inequality in $R(K)\leq F(K)$ may hold for some virtual knots. As an example, virtual knot shown in Fig.~\ref{RASISLESS} has $F(K)=2$, while $R(K)=1$.
\end{remark}
\begin{figure}[H]
\centering
  \includegraphics[width=2.5cm,height=4.5cm]{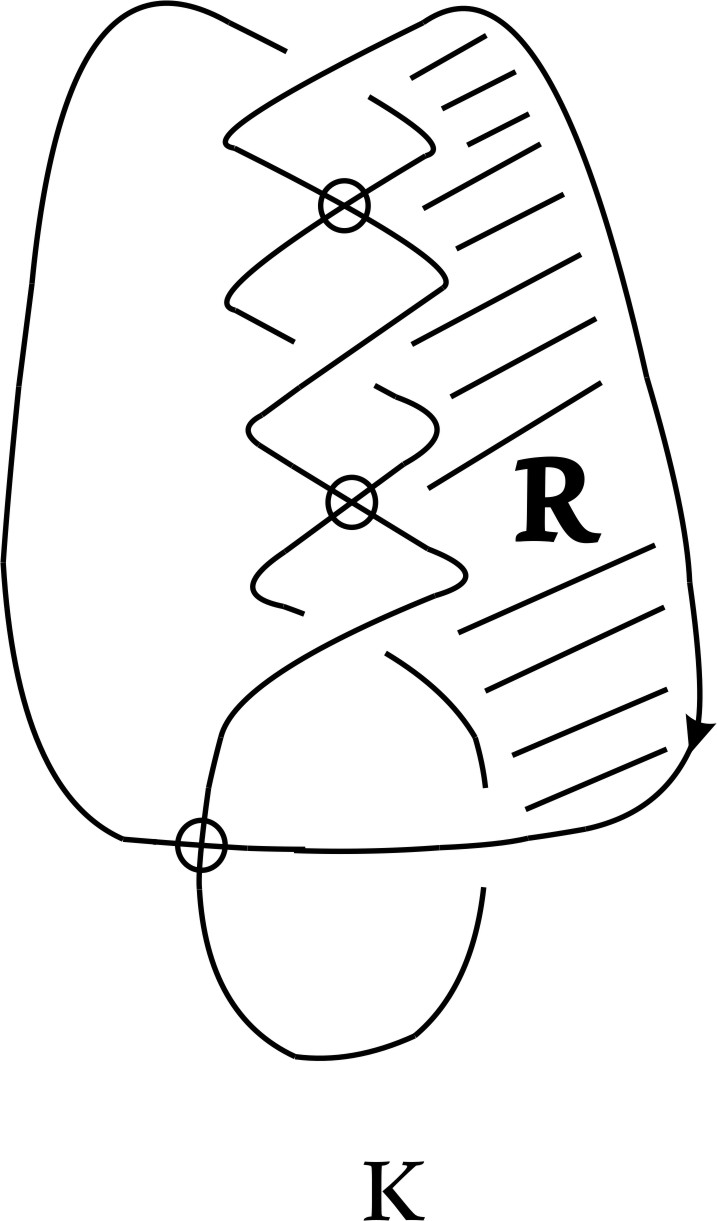}
 \caption{ }
  \label{RASISLESS}
\end{figure}

\noindent One of the important consequence of forbidden moves is the forbidden detour move, denoted by FD, shown in Fig.~\ref{FDETOUR} . It was proved in \cite{kanenobu2001forbidden} that it needs both the forbidden moves $F_u$ and $F_o$ to realize FD including some generalized Reidemeister moves. FD move has the affect of moving across an adjacent head with tail in the Gauss diagram corresponding to a virtual knot diagram, see Fig.~\ref{FDETOUR}.

\begin{figure}[H]
\centering
  \includegraphics[width=9cm,height=1.5cm]{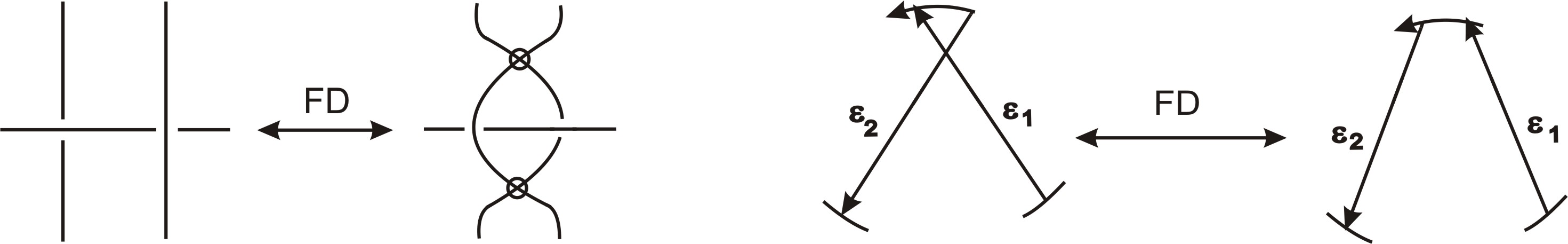}
 \caption{Forbidden detour move }
  \label{FDETOUR}
\end{figure}

\noindent We realize FD move in a single region arc shift operation at a specific region R in the diagram obtained from original diagram by $VR_2$ move. As shown in the Fig.~\ref{FDETOURR}, region $R$ in the diagram obtained by $VR_2$ move in $D$ contains arcs  $\alpha, \beta$ and $\gamma$ as its boundary.
\begin{figure}[H]
\centering
  \includegraphics[width=9cm,height=3.7cm]{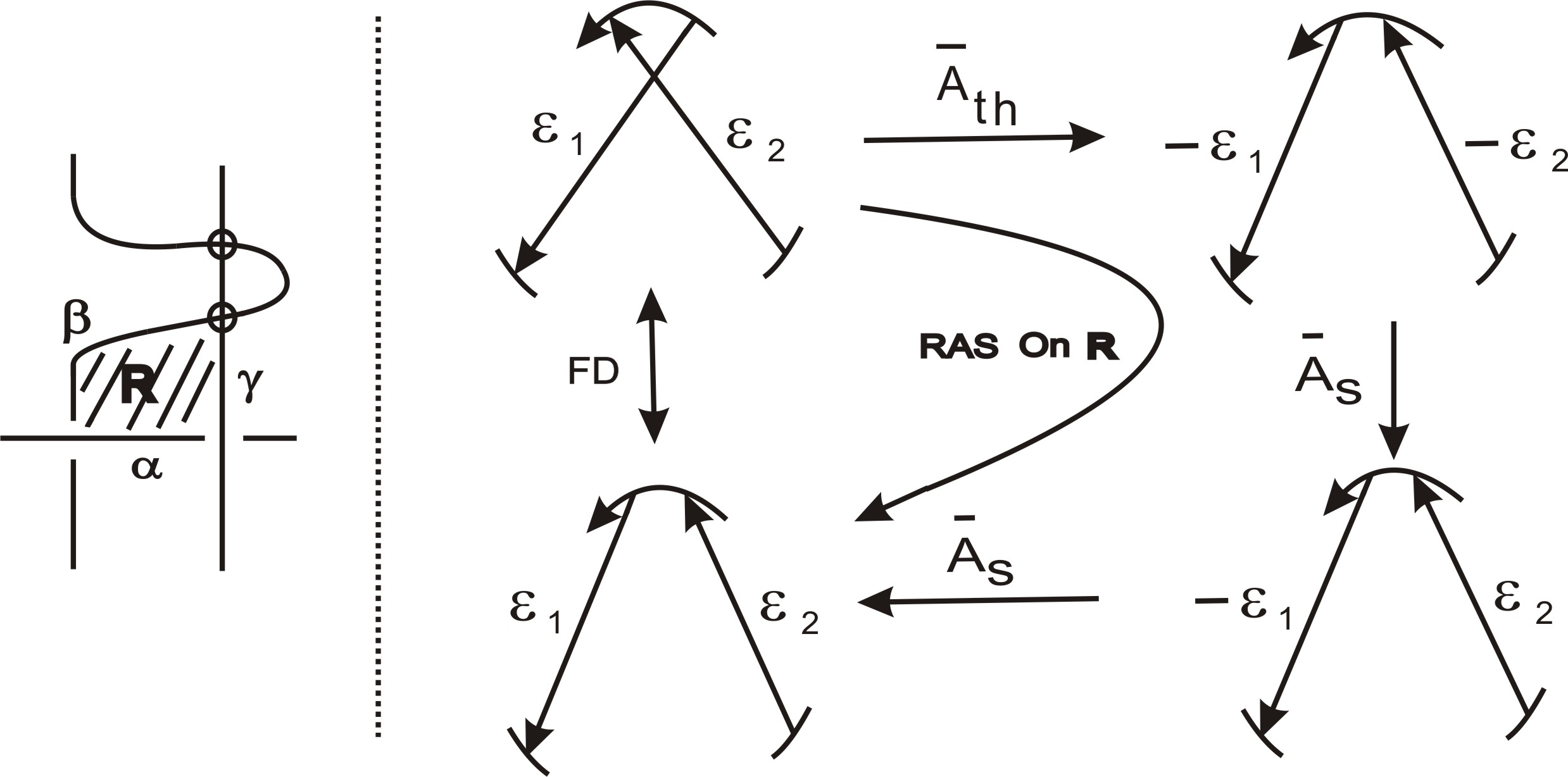}
\caption{FD move via region arc shift at region R }
  \label{FDETOURR}
\end{figure}
Therefore, using the similar argument as in the proof of proposition 8, diagrams obtained by FD move in $D$ and $RAS$ at region $R$ are equivalent virtual knot diagrams and hence the result follows.


\end{document}